\documentclass{amsart}

\usepackage[inner=3cm,outer=3cm]{geometry}
\usepackage{amsmath}
\usepackage{amssymb}
\usepackage{amsthm}
\usepackage{bbm}
\usepackage{mathabx}
\usepackage{url}
\usepackage{float}
\usepackage{graphicx}
\usepackage{caption}
\usepackage{subcaption}
\usepackage{verbatim}
\usepackage{bm}
\usepackage{functan}
\usepackage{enumerate}
\usepackage{enumitem}
\usepackage{hyperref}
\hypersetup{
colorlinks,
citecolor=black,
filecolor=black,
linkcolor=black,
urlcolor=black
}
\allowdisplaybreaks
\usepackage{amsrefs}

\numberwithin{equation}{section}
\newtheorem{theorem}{Theorem}[section]

\newtheorem{lemma}[theorem]{Lemma}

\newtheorem{claim}{Claim}

\newtheorem{proposition}[theorem]{Proposition}

\newcommand{\R}{\mathbb{R}}
\newcommand{\Z}{\mathbb{Z}}
\newcommand{\N}{\mathbb{N}}
\newcommand{\ee}{\mathbf{e}}
\newcommand{\uu}{\mathbf{u}}
\newcommand{\vv}{\mathbf{v}}
\newcommand{\vvn}{\mathbf{v}_n}
\newcommand{\NN}{\mathcal{N}}
\newcommand{\NNt}{\NN_\tau}
\newcommand{\LL}{\mathcal{L}}
\newcommand{\XX}{\mathbf{X}}
\newcommand{\eps}{\epsilon}
\newcommand{\half}{\frac{1}{2}}
\newcommand{\tht}{\theta}
\newcommand{\sg}{\sigma}
\newcommand{\dl}{\delta}
\newcommand{\al}{\alpha}
\newcommand{\la}{\lambda}

\newcommand{\f}{\frac}
\newcommand{\F}{\mathcal{F}}
\newcommand{\Sp}{\mathbb{S}}
\newcommand{\na}{\nabla}
\newcommand{\dd}{\partial}
\newcommand{\D}{\Delta}
\newcommand{\x}{\times}
\newcommand{\xsobot}{X^{s_1,b_1}_\tau}

\newcommand{\xsbt}{X^{s,b}_\tau}
\newcommand{\xsobotm}{X^{s_1-1,b_1-1}_\tau}
\newcommand{\xsbtm}{X^{s-1,b-1}_\tau}
\newcommand{\xxsbt}{\XX^{s,b}_\tau}
\newcommand{\xxsobot}{\XX^{s_1,b_1}_\tau}

\newcommand{\xxsb}{\XX^{s,b}}

\newcommand{\xsbm}{X^{s-1,b-1}}

\newcommand{\vp}{\varphi}
\newcommand{\rd}{\sqrt{-\D}}
\newcommand{\uun}{\uu_n}
\newcommand{\tft}{\widetilde{\mathcal{F}}_{\tau}}
\newcommand{\ft}{\mathcal{F}_{\tau}}
\newcommand{\lan}{\langle}
\newcommand{\ran}{\rangle}
\newcommand{\dt}{d_{\tau}}
\newcommand{\EE}{\mathcal{E}}
\newcommand{\ko}{k_1}
\newcommand{\kt}{k_2}
\newcommand{\kth}{k_3}
\newcommand{\T}{\mathcal{T}}

\newcommand{\been}{\bar{\mathbf{e}}_n}

\title{A Splitting Scheme for the Wave Maps Equation at Low Regularity}

\author[K.~Marsden]{Katie Marsden}
\address{Department of Mathematics, University of California, Los Angeles, CA 90095, USA (K.Marsden)}
\email{kmarsden@math.ucla.edu}

\author[F.~Rousset]{Fr\'ed\'eric Rousset}
\address{Laboratoire de Math\'ematiques d'Orsay (UMR 8628), Universit\'e Paris-Saclay, CNRS, 91405 Orsay Cedex, France (F. Rousset)}
\email{frederic.rousset@universite-paris-saclay.fr}

\author[K.~Schratz]{Katharina Schratz}
\address{LJLL (UMR 7598), Sorbonne Universit\'e, UPMC, 4 place Jussieu, 75005 Paris, France (K. Schratz)}
\email{katharina.schratz@sorbonne-universite.fr}

\begin{document}

\begin{abstract}
We prove convergence of a filtered Lie splitting scheme for the wave maps equation with low regularity initial data in dimension 3. The convergence analysis is performed in discrete Bourgain spaces, as has proved fruitful for the low regularity analysis of the equation in the continuous setting. 
An important   difficulty here is that 
the analysis of wave maps at low regularity requires the use of the null structure of the system, this structure
thus has to be preserved at the discrete level to get an effective stable  low regularity scheme. Since the null structure  involves time derivatives,  the scheme has to be designed carefully.
The presence of time derivatives in the nonlinearity  then constitutes the most significant source of numerical error. Nonetheless, we are able to prove convergence  of the scheme for all subcritical initial data in   $H^s$, $s>d/2$.
\end{abstract}

\maketitle

\tableofcontents
\section{Introduction}
This article concerns the numerical approximation of three dimensional wave maps into the sphere, described by
\begin{equation}\label{WM}
\begin{cases}
(\dd_t^2-\D)u=u\dd^\al u\cdot\dd_\al u\equiv-u(|\dd_tu|^2-|\na_xu|^2)\\
(u,\dd_t u)|_{t=0}=(u_0,v_0)
\end{cases}
\quad\quad (u:\R_t\times\R_x^3\rightarrow \Sp^2).
\end{equation}
Wave maps have been studied numerically via finite element methods in \cite{B2009,BFP2007,BLP2007}, yielding convergence to weak solutions. Weak solutions were further studied in \cite{KW2013} using a reformulation of \eqref{WM} as a first order system via the angular momentum variable $w:=u\times \dd_tu$. In this article we study a semi-discrete Lie splitting scheme for \eqref{WM} and prove convergence to \textit{strong }(Duhamel) solutions, with essentially minimal regularity assumptions on the initial data. Details of the scheme are described in Section \ref{SS_sec} and the main result is contained in Theorem \ref{thm:main_thm}.

In the continuous framework, the wave maps equation has attracted a great deal of attention due to its rich structure and usefulness as a model geometric wave equation. In particular, the question of local wellposedness is well understood. For $s>d/2+1$, local well-posedness in $H^s(\R^d)$ follows from classical energy arguments. This may be improved to $s>\max\{\f{d+1}{2},\f{d+5}{4}\}$ using Strichartz estimates, and the full subcritical range $s>d/2$ was obtained by Klainerman and Machedon \cite{klainerman1996smoothing} and Selberg \cite{selberg1999multilinear} by exploiting the null structure of the nonlinearity in Bourgain spaces (see Section \ref{subsec:background}). This is essentially the sharp threshold for local well-posedness, $s=d/2$ being the scaling critical space for the equation.

Numerically, it is not difficult to establish convergence of the (filtered) Lie splitting scheme for data in $H^s(\R^d)$, $s>d/2+1$, using energy methods, however the extension to low regularity is significantly more delicate. This is the main goal of the present article. There has recently been a large amount of interest in numerical schemes for dispersive equations at low regularity; in particular, both tools suitable to analyze the convergence at low regularity and new schemes with improved local error at low regularity have been introduced. For works on the   nonlinear Schr\"odinger equations, we refer to \cite{Ignat}, \cite{LunOs24}, \cite{LunOs25}, \cite{Li-Wu2021}, \cite{ostermann2022error}, \cite{ostermann2021error}, \cite{ostermann2022fourier}, \cite{Wu2022}; for works on the KDV equation to \cite{HoS}, \cite{LiWu25}, \cite{LiWu21}, \cite{RoS22}, \cite{WuZhao22}; and for works on nonlinear wave equations to \cite{Cao25}, \cite{RoS21}, 
\cite{ruff25}, \cite{ruff}, for example. Most pertinent to us is the article \cite{ostermann2022fourier} in which discrete Bourgain spaces are introduced to perform the convergence analysis, and the article \cite{ruff}, where Ruff and Schnaubelt establish low regularity convergence of a Lie splitting scheme for the nonlinear wave equation via discrete Strichartz estimates.

The main part of this article consists of developing a notion of discrete Bourgain spaces for the wave equation and establishing the necessary multilinear estimates. The Strichartz estimates proved in \cite{ruff} play an important role here. Aside from this, there are two key difficulties in studying wave maps at low regularity. The first is that the null structure of the nonlinearity plays a key role, and care must be taken to preserve this on the discrete level. The second is that the time derivatives in the nonlinearity create an additional source of error in the (temporal) discretization. 
We shall thus introduce and analyze a modified Lie splitting scheme which shows a suitable  null structure at the discrete level.
The second  issue is partially mitigated by enforcing a strong CFL condition, however this remains the largest source of error throughout our analysis.

We remark that our scheme does not preserve the sphere constraint $\uu(t,x)\in\Sp^2$. However, since we establish convergence in $H^s$, $s>d/2$, the solutions approach the sphere uniformly as the time step is taken to zero. There are previous works on numerical wave maps in which the sphere constraint is preserved, however their methods do not straightforwardly apply at low regularity. In \cite{BFP2007} the authors use a projection to the sphere on each time step, which requires a significant amount of smoothness to control. On the other hand the method presented in \cite{BLP2007,KW2013} is incompatible with the null structure.

In the remainder of this introduction we briefly review the relevant theory of \eqref{WM}, then define the scheme and state the main convergence result.

\subsection{Theoretical background}\label{subsec:background}
For a generic quadratic derivative wave equation scaling like \eqref{WM}, one does not expect wellposedness in low regularity subcritical spaces. For example, Lindblad \cite{lindblad1993sharp} showed illposedness of the nonlinear wave equation
$$
\Box u=(\dd_t u)^2
$$
in $H^{2-\eps}(\R^3)\times H^{1-\eps}(\R^3)$ for any $\eps>0$. There is a particular structure in the wave maps nonlinearity which counteracts the effect of the derivatives and allows for stronger wellposedness results. This is known as the null structure, encapsulated in the identity
\begin{equation}\label{eqn:NS}
|\dd_tu|^2-|\na_xu|^2=\Box(u\cdot u)-2u\cdot \Box u
\end{equation}
which reveals cancellations between linear waves propagating in parallel along the lightcone.

The null structure can be exploited by working in the Bourgain spaces (also called hyperbolic Sobolev spaces, see e.g. \cite[Chapter 2.3]{geba2016introduction}), defined by
$$
\|u\|_{X^{s,b}}:=\|\langle |\xi|+|\sg|\rangle^s\langle |\sg|-|\xi|\rangle^b\,\widetilde{u}(\sg,\xi)\|_{L^2_\sg L^2_\xi(\R\times\R^3)},
$$
where $\widetilde{u}:=\widetilde{\F}u$ denotes the spacetime Fourier transform of $u$. 
These spaces are at the regularity of $H^s$ in the sense that
\begin{equation}\label{eqn:embedding}
X^{s,b}\hookrightarrow C^0_t H^s_x
\end{equation}
for any $b>1/2$, $s+b>3/2$ \cite[Proposition 2.7]{geba2016introduction},
and there is the straightforward relation
\begin{equation}\label{eqn:box}
\|\Box u\|_{X^{s-1,b-1}}\lesssim \|u\|_{X^{s,b}}\hspace{3em}(s,\,b\in\R).
\end{equation}
For compact time intervals, the local Bourgain space is defined as the quotient
$$
X^{s,b}([0,T]):=X^{s,b}(\R)\Big/\{u=0\text{ a.e. on } [0,T]\},
$$
with norm $\|u\|_{X^{s,b}([0,T])}=\inf\{\|v\|_{X^{s,b}}:v=u \text{ a.e. on } [0,T]\}$.

Let us now re-cast equation \eqref{WM} as a first order system. Defining $\uu:=(u,\dd_t u)^T$, $\uu_0=(u_0,v_0)^T$, the wave maps equation \eqref{WM} is equivalent to
\begin{equation}
\dd_t\uu =
\begin{pmatrix}
0 & 1\\
\D & 0
\end{pmatrix}\uu
+
\begin{pmatrix}
0\\
-u(|\dd_tu|^2-|\na u|^2)
\end{pmatrix} 
=: 
\begin{pmatrix}
0 & 1\\
\D & 0
\end{pmatrix}
\uu
+\NN(u)
\label{WM3}
\end{equation}
with $\uu(0)=\uu_0$. We will often use $\NN(u)$ to mean only the second component of the nonlinearity vector.

Defining the free evolution operator
\begin{equation*}
\LL_t:=
\begin{pmatrix}
	\cos(t\rd) & \f{\sin(t\rd)}{\rd}\\
	-\rd\sin(t\rd) & \cos(t\rd)
\end{pmatrix},
\end{equation*}
we see that \eqref{WM3} admits the Duhamel formulation
\begin{equation}\label{eqn:mildWM}
\uu(t)=\LL_t\uu_0+\int_0^t\LL_{t-s}\,\NN(u(s))ds.
\end{equation}

The local existence theory for near-critical wave maps is summarized in the following theorem. We consider initial data $\uu_0\in (p+H^s)\x H^{s-1}$, for some fixed $p\in\R^3$.

Let $\mathbf{X}^{s,b}=X^{s,b}\x X^{s-1,b}$ and $p+\mathbf{X}^{s,b}=(p+X^{s,b})\x X^{s-1,b}$.
\begin{theorem}[\cite{klainerman1996smoothing,MR1901147}]\label{lwp}
Let $s,b\in\R$ with $1/2>s-3/2>b-1/2>0$. Let $p\in \R^3$, $\uu_0\in (p+H^s(\R^3))\x H^{s-1}(\R^3)$. There exists $T_\text{max}>0$ such that \eqref{eqn:mildWM} admits a unique maximal solution $\uu\in p+\XX^{s,b}_{\text{loc}}([0,T_{\text{max}}))$.

If in fact $u_0:\R^3\rightarrow\Sp^2$, $v_0:\R^3\rightarrow T\Sp^2$, the solution takes values in the sphere for all $t\in[0,T_{\text{max}})$.
\end{theorem}
Uniqueness is in the sense that $\uu$ is unique in $p+\XX^{s,b}([0,R])$ for any $0\leq R< T_{\text{max}}$. We use this convention throughout the article.

For small initial data, the proof of Theorem \ref{lwp} is by Picard iteration in $p+\XX^{s,b}([0,T])$. The key step is controlling the nonlinearity in $X^{s-1,b-1}$, which is achieved by repeated use of the algebra estimates
\begin{equation}
\|u\cdot v\|_{X^{s,b}}\lesssim\|u\|_{X^{s,b}}\|v\|_{X^{s,b}},\hspace{2em}\|u\cdot F\|_{X^{s-1,b-1}}\lesssim\|u\|_{X^{s,b}}\|F\|_{X^{s-1,b-1}}\label{eqn:algebra-estimates}
\end{equation}
with the identity \eqref{eqn:NS} and property \eqref{eqn:box}. See \cite{geba2016introduction} for a textbook treatment of this case. The large data problem is more subtle; the reader may consult the classical works \cite{MR1901147,selberg1999multilinear}, or Section 10.1 of \cite{marsden2024global}.

\subsection{The scheme}\label{SS_sec}
Fix $\tau\in(0,1)$. In the previous section we discussed the importance of the null structure in the low regularity analysis of \eqref{WM}. Unfortunately, this structure is not preserved on the numerical level. Indeed, letting $\Box_\tau$ denote the semi-discrete wave operator,
$$
\Box_\tau u_n:=\f{u_n-2u_{n-1}+u_{n-2}}{\tau^2}-\D u_n,
$$
and
$$
D_{m\tau}u_k:=\f{u_k-u_{k-m}}{m\tau},
$$
we have
\begin{align}
&\Box_\tau(u_k\cdot v_k)-\Box_\tau u_k\cdot v_k-u_k\cdot\Box_\tau v_k\nonumber\\
&=2(D_\tau u_k\cdot D_\tau v_{k-1}-\na_x u_k\cdot\na_x v_k)-2D_\tau u_k\cdot D_{2\tau}v_k+2D_\tau u_{k-1}D_{2\tau}v_k.\label{discreteNS}
\end{align}

This produces remainder terms in the discretization of $\NN(u)$ which have the generic structure of $(\dd_t u)^2$, which we do not expect to be able to handle at low regularity. To avoid this issue, we re-write the nonlinearity in the form
\begin{align}
\NN(u(t))=-
\begin{pmatrix}
	0\\
	u\bigl(\Box(u\cdot u)-2u\cdot\Box u\bigr)
\end{pmatrix}
\label{WM2}
\end{align}
before making any approximations. This is a valid reformulation for any $u\in X^{s,b}$.

To write down our scheme we need the Fourier multipliers
\begin{equation}\label{eqn:chi}
\Pi_{\tau^{-1/2}}=\chi\left(\f{-i\dd_x}{(100\tau^{1/2})^{-1}}\right)
\end{equation}
Here $\chi$ is a smooth function equal to $1$ for $|\xi|\leq1/2$ and vanishing for $|\xi|>1$. We will often drop the subscript $\tau^{-1/2}$.

The scheme we employ is the following filtered Lie splitting scheme,
invoking explicit Euler to approximate the nonlinear flow:
\begin{align}
\uu_{n+1}=\LL_\tau\bigl(\uun+\tau\chi_{[2\tau,\infty)}(n\tau)\NN_\tau(u_n)\bigr),\qquad \uu_0=\Pi\uu(0).\label{SS}
\end{align}
Here $\NNt$ is the discrete approximation to the nonlinearity \eqref{WM2},
\begin{equation}\label{eqn:N_tau}
\NN_\tau(u_n):=-\Pi
\begin{pmatrix}
0\\
\Pi u_n\, \bigl(\Box_\tau(\Pi u_n\cdot\Pi u_n)-2\Pi u_n\cdot\Box_\tau\Pi u_n\bigr)
\end{pmatrix}.
\end{equation}

The projection to frequencies $\lesssim \tau^{-1/2}$ may seem unnatural given the wave equation scaling, however this condition is necessary for approximating the discrete time derivatives in the nonlinearity. See Lemmas \ref{lem4} and \ref{lem:nonlinear_lemma}.

Our main result is the following.
\begin{theorem}\label{thm:main_thm}
Let $3/2<s_1<s<2$ with $s-s_1<2-s$. Let $p\in\R^3$. For $\uu_0\in (p+H^s(\R^3)) \x H^{s-1}(\R^3)$ let $\uu\in p+\XX^{s,b}([0,T_{\text{max}}))$ be the exact solution to the wave maps equation given in Theorem \ref{lwp} and let $\uun$ denote the iterates defined in \eqref{SS}. For every $0\leq T<T_{\text{max}}$ there exists $\tau_0(T)>0$ such that for all $0<\tau<\tau_0$ it holds
\begin{equation}\label{eqn:rate}
\|\uu(t_n)-\uun\|_{H^{s_1}\x H^{s_1-1}}\leq C(T,s,s_1,\uu_0)\tau^{(s-s_1)/2}
\end{equation}
whenever $0\leq t_n:= n\tau\leq T$.
\end{theorem}

We remark that the upper bound $s<2$ coincides with the lower bound until which Strichartz arguments suffice in the wellposedness theory. We therefore expect that the result extends to all $s>3/2$ using discrete Strichartz estimates. We expect our arguments to generalize to $d\geq3$, and will present results for arbitrary $d\geq2$ where applicable.

\begin{center}
	\begin{figure}[t]
		\includegraphics[scale=0.15]{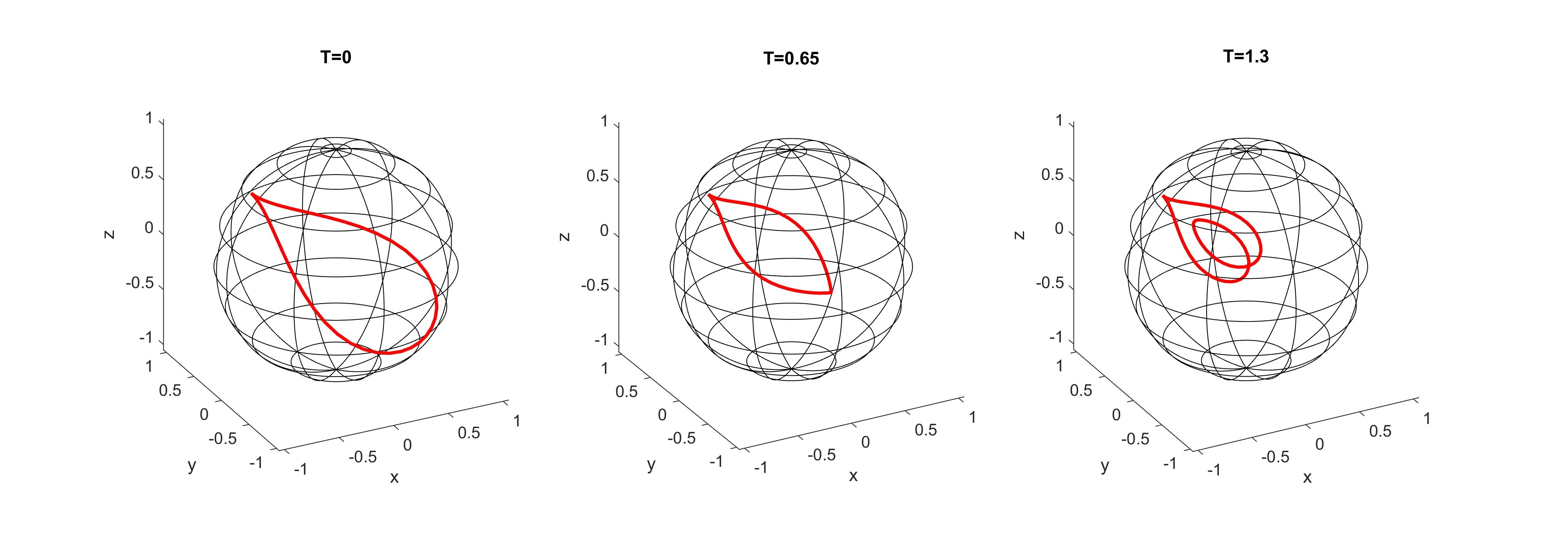}
		\caption{Numerical simulation of 1D wave maps by filtered Lie splitting. The initial data is $u_0(x)=\bigl(1/\sqrt{2}(-\cos\tht(x)+\sin\tht(x)\cdot\sin\phi(x)),-\sin\tht(x)\cdot\cos\phi(x),1/\sqrt{2}(\cos\tht(x)+\sin\tht(x)\cdot\sin\phi(x))\bigr)$, $v_0(x)=(0,0,0)$ for $\tht(x)=2e^{-x^2}$, $\phi(x)=xe^{-x^2}$.}
	\end{figure}
\end{center}

The rest of this paper is organized as follows. In Section \ref{sec:discrete-Bourgain-spaces} we define the discrete wave Bourgain spaces and establish corresponding linear estimates. The wave Bourgain spaces are more complicated to study than their Schr\"odinger analogue due to the symmetric treatment of time and space. Sections \ref{sec:bilinear-estimates} and \ref{sec:trilinear-estimates} concern the development of nonlinear estimates based on Littlewood-Paley decompositions in the frequency and modulation variables. In order to adapt arguments from the continuous setting we must carefully study the geometry of discrete wave interactions in Fourier space, and this comprises the main novelty of these sections. Having established the Bourgain space framework, we return to the question of convergence in Section \ref{sec:error_representation}. Here we use continuous estimates to show that the frequency truncation causes an acceptable error in the approximation, and isolate the local error from the expression which remains. In Section \ref{sec:local-error} we bound the local error using the nonlinear estimates established in Sections \ref{sec:bilinear-estimates} and \ref{sec:trilinear-estimates}, and in Section \ref{section:main_theorem} we complete the proof of Theorem \ref{thm:main_thm}. In Section \ref{section:numerics}, we present the results of numerical experiments.

\section{Discrete Bourgain spaces and linear estimates}\label{sec:discrete-Bourgain-spaces}
In this section we develop a notion of discrete Bourgain spaces for the wave operator and establish linear estimates. The results of this section are valid on $\R\times\R^d$ for any $d\geq1$, and all implicit constants may depend on $d$. 

Assume $0<\tau<1$.
The discrete time Fourier transform of a sequence $u_n=(u_n)_{n\in\Z}$ is given by
$$
\F_\tau(u_n)(\sg):=\tau\sum_{n\in\Z}u_ne^{-in\tau\sg}, \hspace{3em}\sg\in[-\pi/\tau,\pi/\tau],
$$
with the inversion formula
\begin{equation}\label{eqn:inversion}
u_n=\f{1}{2\pi}\int_{-\pi/\tau}^{\pi/\tau}e^{in\tau\sg}\F_\tau(u_n)(\sg)d\sg.
\end{equation}
For a sequence $u_n=(u_n(x))_{n\in\Z,\,x\in\R^d}$ we define the semidiscrete Fourier transform
$$
\tft(u_n)(\sg,\xi):=\tau\sum_{n\in\Z}\hat{u}_n(\xi)e^{-in\tau\sg}, \hspace{3em}\hat{u}_n(\xi)=\int_{x\in\R^d}u_n(x)e^{-ix\cdot\xi}\,dx,
$$
for $\xi\in\R^d,\ \sg\in[-\pi/\tau,\pi/\tau]$.
We will sometimes use $\widetilde{u}_n$ to denote $\tft(u_n)$. Using the notation
$$
\|u_n\|_{\ell^p_\tau}^p=\tau\sum_{n\in\Z}|u_n|^p.
$$
we have Parseval's identity
\begin{equation*}\label{parseval}
\|u_n\|_{\ell^2_\tau L^2_x}=(2\pi)^{-(d+1)}\|\tft(u_n)\|_{L^2_{\sg,\xi}([-\pi/\tau,\pi/\tau]\times\R^d)}.
\end{equation*}
We then define the discrete Bourgain norm
\begin{equation}\label{eqn:Bourgain_norm}
\|u_n\|_{X^{s,b}_{\tau}}:=\|\langle \dt(|\sg|+|\xi|)\ran^{s}\lan\dt(|\sg|-|\xi|)\ran^{b}\tft(u_n)(\sg,\xi)\|_{L^2_{\sg,\xi}([-\pi/\tau,\pi/\tau]\times\R^d)}.
\end{equation}
Here $\dt$ is the $2\pi/\tau$--periodic function
$$
\dt(x):=\f{e^{ix\tau}-1}{\tau}=ie^{ix\tau/2}\f{\sin(x\tau/2)}{(x\tau/2)}x.
$$
To avoid cancellations due to periodicity we restrict to functions with frequency support in $\{|\xi|\leq \pi/2\tau\}$, so that $|\dt(|\sg|\pm|\xi|)|\simeq\bigl||\sg|\pm|\xi|\bigr|$ in the support of the integral \eqref{eqn:Bourgain_norm}, and in particular
$$
|\dt(|\sg|+|\xi|)|\simeq| \dt(\sg)|+| \dt(\xi)|
$$
We remark that the norm \eqref{eqn:Bourgain_norm} is slightly more difficult to work with than the norm introduced in \cite{ostermann2022fourier} since the weights are not periodic in the $\sg$ variable. Our choice reflects the symmetry in the space and time variables and is thus convenient for establishing linear estimates.
As in the continuous case, we define the local discrete Bourgain spaces via
\begin{equation*}
	\|u_n\|_{X^{s,b}_\tau([0,T])}=\inf\{\|v_n\|_{X^{s,b}_\tau}:v_n=u_n \text{ for all } n\tau \in [0,T]\}.
\end{equation*}

The discrete Bourgain spaces enjoy many of the same properties as their continuous counterparts, as shown in the following lemma. We refer the reader to \cite[Section 2.3]{geba2016introduction} for the analogous results in the continuous case, and to \cite{ostermann2022fourier} for the discrete case in the context of the Schr\"odinger equation.

Let $\xxsbt:=\xsbt\x X^{s-1,b}_\tau$.
\begin{lemma}\label{lem:bourgain_prop}
Let $(\uu_n)_n$ be a sequence of functions with Fourier transform supported in $\{|\xi|\leq\pi/2\tau\}$ and $\eta$ be a Schwarz function. The following hold with implicit constants depending only on $\eta,\,s$ and $b$.
\begin{enumerate}[label=(\Roman*)]
	\item 
	For $s\geq1$, $b>1/2$,
	\begin{equation}\label{item1}
		\|\uu_n\|_{\ell^\infty_\tau (H^s\times H^{s-1})}\lesssim\|\uu_n\|_{\xxsbt}.
	\end{equation}
	\item\label{item:II} For $s,\,b\in\R$,
	\begin{equation}\label{eqn:time_cut_off_continuity}
	\|\eta(n\tau)u_n\|_{\xsbt}\lesssim\|u_n\|_{\xsbt}.
	\end{equation}
	\item For $s,\,b\in\R$, $s\geq1$,
	\begin{equation}\label{item2}
		\|\eta(n\tau)\LL_{n\tau}\uu_0\|_{\xxsbt}\lesssim \|\uu_0\|_{H^s\x H^{s-1}}.
	\end{equation}
	\item Set
	$$
	\mathbf{v}_n(x):=\tau\sum_{m=0}^n\LL_{(n-m)\tau}
	\begin{pmatrix}
		0\\
		F_m(x)
	\end{pmatrix}.
	$$
	Then for $b>1/2$, $s+b>3/2$, it holds
	\begin{equation}\label{item3}
		\|\eta(n\tau)\mathbf{v}_n\|_{\xxsbt}\lesssim\|F_n\|_{X^{s-1,b-1}_\tau}.
	\end{equation}
\end{enumerate}
The analogous estimates in the continuous Bourgain spaces are also all true.
\end{lemma}
Note that in contrast to \cite{ostermann2022fourier} we must impose a frequency truncation even to obtain the linear estimates above.

There is a slight subtlety in interpreting the inhomogeneous estimates in the discrete setting, in terms of how to define sums with negative limits. Our definition is as follows:
\begin{align}
	\sum_{k=0}^n f_k:=
	\begin{cases}
		f_0+f_1+\ldots+f_n &\text{if }n\geq0\\
		0&\text{if }n=-1\\
		-(f_{n+1}+f_{n+2}+\ldots+f_{-1})&\text{if }n\leq -2.
	\end{cases}\label{eqn:summation_convention}
\end{align}
We choose this convention so that the identity
\begin{equation}\label{eqn:geometric_prog}
	\sum_{k=0}^n r^k=\f{r^{n+1}-1}{r-1}
\end{equation}
holds for all $n\in\Z$, $r\neq1$. In applications, we will only study sums of functions supported in $\{k\geq0\}$.
\begin{proof}
	The continuous estimates are well-known, we will only prove the discrete versions.
	
(I) It suffices to show that $\|u_n\|_{\ell^\infty_nH^s_x}\lesssim\|u_n\|_{X^{s,b}_\tau}$ for any $s\geq0$. Fix $n\in\Z$. We use the partial inversion formula \eqref{eqn:inversion} and Cauchy-Schwarz to estimate
	\begin{align*}
		\|u_n\|_{H^s_x}^2&=\int_{\xi}\langle\xi\rangle^{2s}\biggl|\int_{-\pi/\tau}^{\pi/\tau}\tft(u_n)(\sg,\xi)e^{in\tau\sg}d\sg\biggr|^2d\xi\\
		&\lesssim \int_\xi\langle\xi\rangle^{2s}\biggl(\int_{-/\pi\tau}^{\pi/\tau} \langle\dt(|\sg|-|\xi|)\rangle^{2b}|\tft(u_n)(\sg,\xi)|^2d\sg\biggr)\biggl(\int_{-\pi/\tau}^{\pi/\tau}\langle\dt(|\sg|-|\xi|)\rangle^{-2b}d\sg\biggr)d\xi
	\end{align*}
	from which the result follows under the frequency constraint since $|\dt(|\sg|-|\xi|)|\gtrsim \bigl||\sg|-|\xi|\bigr|$.
	
	(II) First observe the elementary bound
	$$
	\lan \sg_1\ran^{-1}\lesssim\f{\lan\dt(|\sg|\pm|\xi|)\ran}{\lan\dt(|\sg-\sg_1|\pm|\xi|)\ran}\lesssim\lan\sg_1\ran,\hspace{3em} \text{ for }|\sg|,\,|\sg_1|\leq\pi/\tau,\ |\xi|\leq \pi/2\tau,
	$$
	which can be proved  by separately considering the cases $|\sg_1|\geq\pi/10\tau$ and $|\sg_1|\leq\pi/10\tau$.
	
	By the partial inversion formula we have
	\begin{equation}\label{eqn:point2}
	\|\eta(n\tau)u_n\|_{\xsbt}\leq\f{1}{2\pi}\int_{-\pi/\tau}^{\pi/\tau}|\F_{\tau}(\eta)(\sg_1)|\,\|e^{in\tau\sg_1}u_n\|_{X^{s,b}_\tau}\,d\sg_1,
	\end{equation}
	where
	\begin{align}
	&\|e^{in\tau\sg_1}u_n\|_{X^{s,b}_\tau}^2\nonumber\\
	&=\int_{\sg=-\pi/\tau}^{\pi/\tau}\int_\xi \langle \dt(|\sg|+|\xi|)\ran^{2s}\lan\dt(|\sg|-|\xi|)\ran^{2b}|\tft(u_n)(\sg-\sg_1,\xi)|^2d\sg d\xi\nonumber\\
	&\lesssim\lan\sg_1\ran^{2(|s|+|b|)}\int_{\sg=-\pi/\tau}^{\pi/\tau}\int_\xi\langle \dt(|\sg-\sg_1|+|\xi|)\ran^{2s}\lan\dt(|\sg-\sg_1|-|\xi|)\ran^{2b}|\tft(u_n)(\sg-\sg_1,\xi)|^2d\sg d\xi\nonumber\\
	&\lesssim\lan\sg_1\ran^{2(|s|+|b|)}\int_{\sg=-\pi/\tau-\sg_1}^{\pi/\tau-\sg_1}\int_\xi\langle \dt(|\sg|+|\xi|)\ran^{2s}\lan\dt(|\sg|-|\xi|)\ran^{2b}|\tft(u_n)(\sg,\xi)|^2d\sg d\xi.\label{eqn:shifted_Xsb_norm}
	\end{align}
	Suppose now $\sg_1\geq0$. Applying the estimate
	$$
	\lan\sg_1\ran^{-1}\lesssim \f{\lan\dt(|\sg|\pm|\xi|)\ran}{\lan\dt(|\sg+2\pi/\tau|\pm|\xi|)\ran}\lesssim\lan\sg_1\ran \quad(|\sg_1|\leq\pi/\tau,\ |\xi|\leq \pi/2\tau \text{ and } |\sg\pm\pi/\tau|\leq|\sg_1|)
	$$
	in the region $\sg\in[-\pi/\tau-\sg_1,-\pi/\tau]$, alongside the periodicity of $\tft(u_n)$, we see that
	\begin{equation}
	\|e^{in\tau\sg_1}u_n\|_{\xsbt}^2\lesssim\lan\sg_1\ran^{4(|s|+|b|)}\|u_n\|_{\xsbt}^2.\label{eqn:bound_on_shifted_Xsb}
	\end{equation}
	The same result holds for $\sg_1\leq 0$.
	Returning to \eqref{eqn:point2}, the proof is complete thanks to the rapid decay of $\tft(\eta)$,
	\begin{equation}\label{eqn:decay-F(eta)}
	|\F_\tau(\eta)(\sg_1)|\lesssim_N|\dt(\sg_1)|^{-N},
	\end{equation}
	which holds uniformly in $\tau$ \cite[Equation (41)]{ostermann2022fourier}.
	
	(III) 
	First take $\uu_0=(u_0,0)$. A direct computation shows
	$$
	\tft\bigl(\eta(n\tau)\LL_{n\tau}(u_0,0)\bigr)(\sg,\xi)=\f{1}{2i}\begin{pmatrix}
		i\bigl(\ft{\eta}(\sg+|\xi|)+\ft{\eta}(\sg-|\xi|)\bigr)\hat{u}_0(\xi)\\
		|\xi|\bigl(\ft{\eta}(\sg+|\xi|)-\ft{\eta}(\sg-|\xi|)\bigr)\hat{u}_0(\xi)\end{pmatrix}.
	$$
	Using \eqref{eqn:decay-F(eta)} and the fact that $|\xi|\leq\pi/2\tau$ we have
	\begin{align*}
		\|\eta(n\tau)\LL_{n\tau}(u_0,0)\|_{\xxsbt}
		&\lesssim\|\lan\dt(|\sg|+|\xi|)\ran^s\lan\dt(|\sg|-|\xi|)\ran^b\ft\eta(\sg\pm|\xi|)\hat{u}_0(\xi)\|_{L^2_{\sg,\xi}}\\
		&\lesssim_N\|\lan\dt(|\sg|+|\xi|)\ran^s\lan\dt(|\sg|-|\xi|)\ran^b\lan\dt(\sg\pm|\xi|)\ran^{-N}\hat{u}_0(\xi)\|_{L^2_{\sg,\xi}}.
	\end{align*}
	Crudely bounding $\lan\dt(|\sg|+|\xi|)\ran\lesssim\lan\xi\ran\lan\dt(|\sg|-|\xi|)\ran$ and $|\dt(|\sg|\pm|\xi|)|\gtrsim|\dt(|\sg|-|\xi|)|$ in the support of the integral we obtain
	\begin{align*}
		\|\eta(n\tau)\LL_{n\tau}(u_0,0)\|_{\XX^{s,b}_{\tau}}&\lesssim_N\|\lan\xi\ran^s\lan\dt(|\sg|-|\xi|)\ran^{s+b-N}\hat{u}_0(\xi)\|_{L^2_{\sg,\xi}}.
	\end{align*}
	The result follows upon choosing $N$ sufficiently large.
	
	The case $\uu_0=(0,v_0)$ can be treated similarly, distinguishing the cases $|\xi|<\half$ and $|\xi|>\half$ in order to handle the singularity in the operator $\LL_{n\tau}$ at $\xi=0$.
	
	(IV) To avoid repetition, we will only study the first component of $\vv_n$:
	$$
	v_n=\tau\sum_{m=0}^n\f{\sin((n-m)\tau\rd)}{\rd}F_m(x)
	$$
	Writing $\hat{F}_m(\xi)$ using the partial inversion formula \eqref{eqn:inversion} and then applying the identity \eqref{eqn:geometric_prog}, we calculate
	\begin{align}
		&\hat{v}_n(\xi)=\tau\sum_{m=0}^n\f{e^{i(n-m)\tau|\xi|}-e^{-i(n-m)\tau|\xi|}}{2i|\xi|}\cdot\f{1}{2\pi}\int_{-\pi/\tau}^{\pi/\tau}\widetilde{F}(\sg_1,\xi)e^{im\tau\sg_1}d\sg_1\nonumber\\
		&=-\f{i}{4\pi}\int_{-\pi/\tau}^{\pi/\tau}\left[\f{(e^{i(n+1)\tau\sg_1}-e^{i(n+1)\tau|\xi|})e^{-i\tau|\xi|}}{d_\tau(\sg_1-|\xi|)}-\f{(e^{i(n+1)\tau\sg_1}-e^{-i(n+1)\tau|\xi|})e^{i\tau|\xi|}}{d_\tau(\sg_1+|\xi|)}\right]\f{\widetilde{F}(\sg_1,\xi)}{|\xi|}d\sg_1.\label{eqn:vn-hat}
	\end{align}
	
	Given this formula, and the restriction $|\xi|\leq \pi/2\tau$ which ensures that $|d_\tau(\sg_1\pm|\xi|)|\simeq |\sg_1\pm|\xi||$ we may now proceed almost exactly as in \cite[Theorem 2.10]{geba2016introduction}. 
	
	We decompose $v_n=v_n^{(1)}+v_n^{(2)}$, where
	$$
	\hat{v}_n^{(1)}(\xi)=-\f{i}{4\pi}\int_{\substack{|\sg_1|\leq\pi/\tau\\|\sg_1|+|\xi|<\half}}\ldots d\sg_1
	$$
	and $v_n^{(2)}$ is the remainder.
	
	We start by considering $v_n^{(1)}$. Again set $g(\sg):=\F_\tau(\eta(n\tau))(\sg)$. From \eqref{eqn:vn-hat} we have
	\begin{align*}
		\tft(\eta(n\tau)v_n^{(1)})(\sg,\xi)
		&=-\f{i}{4\pi}\int_{|\sg_1|+|\xi|<\half}\left[\f{e^{-i\tau|\xi|}(e^{i\tau\sg_1} g(\sg-\sg_1)-e^{i\tau|\xi|}g(\sg-|\xi|))}{d_\tau(\sg_1-|\xi|)}\right.\\
		&\hspace{9em}-\left.\f{e^{i\tau|\xi|}(e^{i\tau\sg_1} g(\sg-\sg_1)-e^{-i\tau|\xi|}g(\sg+|\xi|))}{d_\tau(\sg_1+|\xi|)}\right]\f{\widetilde{F}(\sg_1,\xi)}{|\xi|}d\sg_1.
	\end{align*}
	Thanks to the difference structure of the expression in brackets and the rapid decay of $g$ and its derivatives, we may bound
	\begin{align*}
	|\tft(\eta(n\tau)v_n^{(1)})(\sg,\xi)|\lesssim_N\int_{|\sg_1|+|\xi|<\half}|\xi|\lan\sg\ran^{-N}\f{\widetilde{F}(\sg_1,\xi)}{|\xi|}d\sg_1.
	\end{align*}
	for any $N>0$.
	Since $|\xi|<1/2<\pi/2\tau$, we have that $\langle d_{\tau}(|\sg|\pm|\xi|)\rangle\simeq\langle\sg\rangle$. It follows by Cauchy-Schwarz that
	\begin{align*}
		\|\eta(n\tau)v^{(1)}_n\|_{X^{s,b}_\tau}^2&\lesssim_N\int_{|\sg|<\pi/\tau}\int_{|\xi|<\half}\langle\sg\rangle^{2(s+b-N)}\biggl(\int_{|\sg_1|+|\xi|<\half}|\widetilde{F}(\sg_1,\xi)|\,d\sg_1\biggr)^2d\xi\,d\sg\\
		&\lesssim \int_{|\sg_1|+|\xi|<\half}|\widetilde{F}(\sg_1,\xi)|^2d\xi\,d\sg_1
	\end{align*}
	which is more than acceptable given the domain of integration.

	
	We now turn to $v_n^{(2)}$. Rewrite
	\begin{align*}
		\hat{v}_n^{(2)}(\xi)&=\f{i}{4\pi}\int_{|\sg_1|+|\xi|>\half}e^{in\tau\sg_1}\left[e^{i(\sg_1-|\xi|)\tau}\f{(1-e^{-i(n+1)\tau(\sg_1-|\xi|)})}{d_\tau(\sg_1-|\xi|)}\right.\\
		&\hspace{9em}-\left.e^{i(\sg_1+|\xi|)\tau}\f{(1-e^{-i(n+1)\tau(\sg_1+|\xi|)})}{d_\tau(\sg_1+|\xi|)}\right]\f{\widetilde{F}(\sg_1,\xi)}{|\xi|}d\sg_1
	\end{align*}
	then further decompose $v_n^{(2)}=v_n^{(3)}+v_n^{(4)}$ with
	\begin{align}
		\hat{v}_n^{(3)}(\xi)&=\f{i}{4\pi}\int_{|\sg_1|+|\xi|>\half}e^{in\tau\sg_1}\left[e^{i(\sg_1-|\xi|)\tau}\f{(1-\chi_{(0,\infty)}(\sg_1)a(|\sg_1|-|\xi|))}{d_\tau(\sg_1-|\xi|)}\right.\nonumber\\
		&\hspace{9em}-\left.e^{i(\sg_1+|\xi|)\tau}\f{(1-\chi_{(-\infty,0)}(\sg_1)a(|\sg_1|-|\xi|))}{d_\tau(\sg_1+|\xi|)}\right]\f{\widetilde{F}(\sg_1,\xi)}{|\xi|}d\sg_1\label{eqn:int}
	\end{align}
	where $a\in C_c^\infty(\R)$ is supported on $(-1/4,1/4)$ and equal to one in a neighborhood of zero, and $\chi$ denotes a sharp indicator function. Then
	\begin{align*}
		\widetilde{v}^{(3)}(\sg,\xi)&=\f{i}{2}\chi_{|\sg|+|\xi|>\half}\left[e^{i\tau(\sg-|\xi|)}\f{(1-\chi_{(0,\infty)}(\sg)a(|\sg|-|\xi|))}{d_\tau(\sg-|\xi|)}\right.\\
		&\hspace{7em}-\left.e^{i\tau(\sg+|\xi|)}\f{(1-\chi_{(-\infty,0)}(\sg)a(|\sg|-|\xi|))}{d_\tau(\sg+|\xi|)}\right]\f{\widetilde{F}(\sg,\xi)}{|\xi|}.
	\end{align*}
	We will show that
	$$
	|\widetilde{v}^{(3)}_n(\sg,\xi)|\lesssim\lan d(|\sg|+|\xi|)\ran^{-1}\lan d(|\sg|-|\xi|)\ran^{-1}|\widetilde{F}(\sg,\xi)|,
	$$
	which is sufficient given that multiplication by $\eta(n\tau)$ is continuous on $X^{s,b}_\tau$ by \ref{item:II}.
	
	Without loss of generality we restrict to $\sg>0$. For $|\sg|+|\xi|>\half$,
	\begin{align*}
		\widetilde{v}^{(3)}_n(\sg,\xi)&=\left[e^{i\tau(\sg-|\xi|)}\f{(1-a(\sg-|\xi|))}{d_\tau(\sg-|\xi|)}-\f{e^{i\tau(\sg+|\xi|)}}{d_\tau(\sg+|\xi|)}\right]\f{\widetilde{F}(\sg,\xi)}{|\xi|}.
	\end{align*}
	When $|\sg-|\xi||>1/4$, $a(\sg-|\xi|)=0$, and we find
	\begin{align*}
		&|\widetilde{v}^{(3)}_n(\sg,\xi)|\\
		&\lesssim\Bigl(\Bigl|\cos(\tau|\xi|)\Bigl(\f{1}{d_\tau(\sg-|\xi|)}-\f{1}{d_\tau(\sg+|\xi|)}\Bigr)\Bigr|
		+\Bigl|\sin(\tau|\xi|)\Bigl(\f{1}{d_\tau(\sg-|\xi|)}+\f{1}{d_\tau(\sg+|\xi|)}\Bigr)\Bigr|\Bigr)\f{|\widetilde{F}(\sg,\xi)|}{|\xi|}
	\end{align*}
	Then
	\begin{align*}
		\f{1}{d_\tau(\sg-|\xi|)}-\f{1}{d_\tau(\sg+|\xi|)}&=2i\tau e^{i\tau\sg}\f{\sin(\tau|\xi|)}{(e^{i\tau(\sg-|\xi|)}-1)(e^{i\tau(\sg+|\xi|)}-1)}
	\end{align*}
	and since $|\xi|<\pi/2\tau$, $|\sg|\pm|\xi|>1/4$, it follows that
	$$
	\left|\f{1}{d_\tau(\sg-|\xi|)}-\f{1}{d_\tau(\sg+|\xi|)}\right|\lesssim\f{\tau^{-1}\sin(\tau|\xi|)}{\lan d_\tau(\sg+|\xi|)\ran \lan d_\tau(\sg-|\xi|)\ran}
	$$
	as required. 
	A similar estimate for the $\sin$ term completes the case $\sg>0$, $|\sg-|\xi||\geq1/4$.
	
	For $|\sg-|\xi||<1/4$, we use the smoothness of $a$ to estimate
	$$
	|1-a(\sg-|\xi|)|\lesssim|\sg-|\xi||.
	$$
	Then under the assumption $\sg+|\xi|>1/2$, we have $\lan d_\tau(|\sg|+|\xi|)\ran\simeq|\xi|$ and $\lan d_\tau(|\sg|-|\xi|)\ran\simeq1$ and we infer that
	\begin{align*}
		|\widetilde{v}^{(3)}_n(\sg,\xi)|&\lesssim\left(\f{|\sg-|\xi||}{|d_\tau(\sg-|\xi|)|}+\f{1}{|d_\tau(\sg+|\xi|)|}\right)\f{|\widetilde{F}(\sg,\xi)|}{|\xi|}\lesssim\f{|\widetilde{F}(\sg,\xi)|}{\lan d_\tau(|\sg|+|\xi|)\ran\lan d_\tau(|\sg|-|\xi|)\ran}.
	\end{align*}
	This completes the study of $v^{(3)}$.
	
	We omit the analysis of the remaining term $v^{(4)}_n$, for which the argument is similar to that in the continuous case \cite{geba2016introduction} with modifications as above. \qedhere

\end{proof}

The next result shows that we have the expected behavior of the d'Alembertian on the discrete Bourgain spaces, provided we impose a strong frequency restriction.
\begin{lemma}\label{lem4}
Let $u_n\in X^{s,b}_\tau$ with Fourier support contained in $\{|\xi|\leq\pi/(2\tau^{1/2})\}$. For any $s,b\in\R$ it holds
$$
\|\Box_\tau u_n\|_{X^{s-1,b-1}_\tau}\lesssim_{s,b}\|u_n\|_{X^{s,b}_\tau}.
$$
\end{lemma}
\begin{proof}
We write
$$
\F_\tau(\Box_\tau u_n)(\sg,\xi)=(|\xi|+id_\tau(-\sg))(|\xi|-id_\tau(-\sg))\widetilde{u}_n(\sg,\xi)
$$
where for $|\sg|\lesssim\tau^{-1}$, $|\xi|\lesssim\tau^{-1/2}$, Taylor expansion yields
$$
|(|\xi|+i\,d_\tau(-\sg))(|\xi|-i\,d_\tau(-\sg))|\lesssim(|\xi|^2-\sg^2)+\sg^2\tau(|\sg|+|\xi|)+(\sg^2\tau)^2.
$$
The first term here is readily seen to be acceptable. For the remaining two we must separately work in the regimes $|\sg|\leq|\xi|/2$, $|\xi|/2<|\sg|<2|\xi|$ and $|\sg|\geq2|\xi|$. The first case is straightforward to handle since
$$
\sg^2\tau\lesssim|\sg|\lesssim|\xi|\lesssim \lan\dt(|\sg|\pm|\xi|)\ran
$$
and the third case is similar. For the remaining case we bound
$$
\sg^2\tau\lesssim\sg^2/|\xi|^2\lesssim 1\lesssim \lan\dt(|\sg|\pm|\xi|)\ran.
$$
This completes the proof.
\end{proof}

The remaining results in this section are directly related to the local error analysis (Section \ref{sec:local-error}), demonstrating how regularity can be traded for convergence.
\begin{lemma}\label{linear_lemma}
Let $s_1\leq s$, $s-s_1\leq 1$ and $b\in\R$. Then for any $|t|<\tau$ it holds
$$
\|\Pi_{\pi/2\tau}(\LL_t-I)\|_{\xxsbt\rightarrow\XX^{s_1,b}_\tau}\lesssim_{s,s_1,b}\tau^{s-s_1}.
$$
In particular,
$$
\|\Pi_{\pi/2\tau}\LL_t\|_{\xxsbt\rightarrow\XX^{s,b}_\tau}\lesssim_{s,b}1.
$$
\end{lemma}
\begin{proof}
Observe that the operator $\LL_t-I$ corresponds to the matrix
$$
\begin{pmatrix}
	\cos(t|\xi|)-1 & \sin(t|\xi|)/|\xi| \\
	-|\xi|\sin(t|\xi|) & \cos(t|\xi|)-1
\end{pmatrix}
$$
in Fourier space. This is bounded componentwise by
\begin{equation}\label{bound}
	\begin{pmatrix}
		|t\xi|^2 & |t| \\
		|t||\xi|^2 & |t\xi|^2
	\end{pmatrix}.
\end{equation}
It follows that
\begin{align}
	\|\Pi_{\pi/2\tau}(\LL_t-I)\uu_n\|_{\XX^{s_1,b}_\tau}^2
	&\lesssim\int_{-\pi/\tau}^{\pi/\tau}\int_{|\xi|\leq\pi/2\tau}\langle d_\tau(|\sg|+|\xi|)\rangle^{2s_1}d_\tau(|\sg|-|\xi|)\rangle^{2b}|t\xi|^4|\widetilde{u}_n(\sg,\xi)|^2\,d\sg d\xi\label{A}\\
	&\quad+\int_{-\pi/\tau}^{\pi/\tau}\int_{|\xi|\leq\pi/2\tau}\langle d_\tau(|\sg|+|\xi|)\rangle^{2s_1}d_\tau(|\sg|-|\xi|)\rangle^{2b}|t|^2|\widetilde{v}_n(\sg,\xi)|^2\,d\sg d\xi\label{B}\\
	&\quad+\int_{-\pi/\tau}^{\pi/\tau}\int_{|\xi|\leq\pi/2\tau}\langle d_\tau(|\sg|+|\xi|)\rangle^{2(s_1-1)}d_\tau(|\sg|-|\xi|)\rangle^{2b}|t|^2|\xi|^4|\widetilde{u}_n(\sg,\xi)|^2\,d\sg d\xi\label{C}\\
	&\quad+\int_{-\pi/\tau}^{\pi/\tau}\int_{|\xi|\leq\pi/2\tau}\langle d_\tau(|\sg|+|\xi|)\rangle^{2(s_1-1)}d_\tau(|\sg|-|\xi|)\rangle^{2b}|t\xi|^4|\widetilde{v}_n(\sg,\xi)|^2\,d\sg d\xi.\label{D}
\end{align}
For the first term we estimate
$$
|t\xi|^4\lesssim |t|^4 |\xi|^{4-2(s-s_1)}|\xi|^{2(s-s_1)}\lesssim \tau^{2(s-s_1)}|\xi|^{2(s-s_1)}\lesssim\tau^{2(s-s_1)}\lan\dt(|\sg|+|\xi|)\ran^{2(s-s_1)}
$$
which is acceptable.

For \eqref{B} we simply bound $\dt(|\sg|+|\xi|)$ by $\tau^{-1}$ to obtain
$$
\langle d_\tau(|\sg|+|\xi|)\rangle^{2s_1}|t|^2\lesssim \langle d_\tau(|\sg|+|\xi|)\rangle^{2(s-1)}\tau^{2(s-s_1)}.
$$
For \eqref{C} we use
$$
|t|^2|\xi|^4\lesssim |t|^2 |\xi|^{2-2(s-s_1)}|\xi|^{2(1+s-s_1)}\lesssim \tau^{2(s-s_1)}\lan\dt(|\sg|+|\xi|)\ran^{2(1+s-s_1)}
$$
and \eqref{D} is similar.
\end{proof}

Next we have a lemma controlling the embedding of continuous into discrete Bourgain spaces. 
\begin{lemma}\label{cts_dis_lem}
	Let $u$ be a function on $\R\x\R^d$ restricted to frequencies $|\xi|\leq \pi/2\tau$, $s>0$, $b>1/2$. Define $u_n(\cdot):=u(n\tau,\cdot)$. Then
	$$
	\|u_n\|_{X^{s,b}_{\tau}}\lesssim_{s,b}\|u\|_{X^{s,b}}.
	$$
	More generally, it holds
	$$
	\sup_{\theta\in\R}\|u(n\tau+\tht,\cdot)\|_{\xsbt}\lesssim_{s,b}\|u\|_{\xsbt}.
	$$
\end{lemma}
\begin{proof}
	By definition,
	$$
	\tft(u_n)(\sg,\xi)=\tau\sum_{n\in\Z}e^{-in\tau\sg}\hat{u}(n\tau,\xi).
	$$
	By the Poisson summation formula this can be expressed as
	\begin{equation}\label{eqn:poisson_sum}
		\tft(u_n)(\sg,\xi)=\sum_{n\in\Z}\widetilde{u}(\sg+\tfrac{2\pi}{\tau}n,\xi),
	\end{equation}
	so we have to estimate
	$$
	\int_{\sg=-\pi/\tau}^{\pi/\tau}\bigl|\sum_{n\in\Z}\lan\dt(|\sg|+|\xi|)\ran^{s}\lan\dt(|\sg|-|\xi|)\ran^{b}\widetilde{u}(\sg+\tfrac{2\pi}{\tau}n,\xi)\bigr|^2\,d\sg\,d\xi.
	$$
	The $n=0$ contribution is clearly acceptable. For the remaining portion of the sum, we estimate
	\begin{align*}
		&\Bigl|\sum_{n\neq0}\lan\dt(|\sg|-|\xi|)\ran^{b}\widetilde{u}(\sg+\tfrac{2\pi}{\tau}n,\xi) \Bigr|^2\\
		&\lesssim\Bigl(\sum_{n\neq0} \lan|\sg+\tfrac{2\pi}{\tau}n|-|\xi|\ran^{-2b}\Bigr) \Bigl(\sum_{n\neq0}\lan\dt(|\sg|-|\xi|)\ran^{2b}\lan|\sg+\tfrac{2\pi}{\tau}n|-|\xi|\ran^{2b}|\widetilde{u}(\sg+\tfrac{2\pi}{\tau}n,\xi)|^2 \Bigr)
	\end{align*}
	The first sum converges since $b>1/2$, and it is $O(\tau^{2b})$. In the second sum we use the crude bound $|\dt(|\sg|-|\xi|)|\lesssim\tau^{-1}$, and deduce
	\begin{align*}
		\Bigl|\sum_{n\neq0}\lan\dt(|\sg|-|\xi|)\ran^{b}\widetilde{u}(\sg+\tfrac{2\pi}{\tau}n,\xi) \Bigr|^2\lesssim \sum_{n\neq0}\lan|\sg+\tfrac{2\pi}{\tau}n|-|\xi|\ran^{2b}|\widetilde{u}(\sg+\tfrac{2\pi}{\tau}n,\xi)|^2.
	\end{align*}
	Integrating against the appropriate weight we therefore obtain
	\begin{align}
		&\int_{\sg=-\pi/\tau}^{\pi/\tau}\bigl|\sum_{n\neq0}\lan\dt(|\sg|+|\xi|)\ran^{s}\lan\dt(|\sg|-|\xi|)\ran^{b}\widetilde{u}(\sg+\tfrac{2\pi}{\tau}n,\xi)\bigr|^2\,d\sg\,d\xi\nonumber\\
		&\lesssim\sum_{n\neq0}\int_{-\pi/\tau}^{\pi/\tau}\int_{|\xi|\leq\pi/2\tau}\langle \dt(|\sg|+|\xi|)\ran^{2s}\lan|\sg+\tfrac{2\pi}{\tau}n|-|\xi|\ran^{2b}|\widetilde{u}(\sg+\tfrac{2\pi}{\tau}n,\xi)|^2d\sg d\xi\nonumber\\
		&= \sum_{n\neq0}\int_{(2n-1)\pi/\tau}^{(2n+1)\pi/\tau}\int_{|\xi|\leq\pi/2\tau}\lan\dt(|\sg-\tfrac{2\pi}{\tau}n|+|\xi|)\ran^{2s}\lan|\sg|-|\xi|\ran^{2b}|\widetilde{u}(\sg,\xi)|^2d\sg d\xi\label{eqn:sum_bound}
	\end{align}
	By the periodicity of $\dt$,
	$$
	\lan\dt(|\sg-\tfrac{2\pi}{\tau}n|+|\xi|)\ran=\lan\dt(\pm\sg+|\xi|)\ran\lesssim\lan|\sg|+|\xi|\ran,
	$$
	so
	$$
	\eqref{eqn:sum_bound}\lesssim\sum_{n\neq0}\int_{(2n-1)\pi/\tau}^{(2n+1)\pi/\tau}\int_{|\xi|\leq\pi/2\tau}\langle |\sg|+|\xi|\ran^{2s}\lan|\sg|-|\xi|\ran^{2b}|\widetilde{u}(\sg,\xi)|^2d\sg d\xi=\|u\|_{X^{s,b}}^2.
	$$
	The second statement of the lemma is immediate by translation invariance of the continuous Bourgain norm.
\end{proof}

We end this section with an initial estimate for the error produced by the discrete approximation of the wave operator in the nonlinearity.
\begin{lemma}\label{key_lemma}
	Let $u$ be a function defined on $\R\x\R^d$ restricted to frequencies $|\xi|\leq\pi/2\tau$ and set $u_n:=u(n\tau,\cdot)$. 
	Then for any $s,b,s_1,b_1$ with $s+b>3/2$, $b_1<1$, and $s_1>1$, it holds
	$$
	\|(\Box-\Box_\tau)u_n\|_{X^{s_1-1,b_1-1}_\tau}\lesssim_{s,s_1,b,b_1}\tau\|\dd_t^3u\|_{X^{s_1-1,b_1-1}}+\tau^{s-s_1+b}\|\dd_tu\|_{X^{s,b}}.
	$$
	Here $\Box u_n\equiv(\Box u)(n\tau,\cdot)$.
\end{lemma}

\begin{proof}
	Observe that by \eqref{eqn:poisson_sum},
	$$
	\tft\bigl((\Box-\Box_\tau)u_n\bigr)(\sg,\xi)=-\sum_{n\in\Z}\Bigl(\bigl(\sg+\tfrac{2\pi}{\tau}n\bigr)^2+\Bigl(\f{e^{-i\tau\sg}-1}{\tau}\Bigr)^2\Bigr)\widetilde{u}(\sg+\tfrac{2\pi}{\tau}n,\xi).
	$$
	We have to estimate
	\begin{align}
		\|(\Box-\Box_\tau)u_n\|_{X^{s_1-1,b_1-1}_\tau}^2
		&\simeq\int_{-\pi/\tau}^{\pi/\tau}\int_{|\xi|\leq\pi/2\tau}\lan\dt(|\sg|+|\xi|)\ran^{2(s_1-1)}\lan\dt(|\sg|-|\xi|)\ran^{2(b_1-1)}\nonumber\\
		&\quad\cdot \Bigl|\sum_{n\in\Z}\Bigl(\bigl(\sg+\tfrac{2\pi}{\tau}n\bigr)^2+\Bigl(\f{e^{-i\tau\sg}-1}{\tau}\Bigr)^2\Bigr)\widetilde{u}(\sg+\tfrac{2\pi}{\tau}n,\xi)\Bigr|^2 d\sg d\xi.\label{eqn:Box-Box_tau}
	\end{align}
	For $n=0$, we use a Taylor expansion to bound
	\begin{equation*}
		\Bigl|\sg^2+\Bigl(\f{e^{-i\tau\sg}-1}{\tau}\Bigr)^2\Bigr|\lesssim|\tau\sg^3|
	\end{equation*}
	and it is immediate that the $n=0$ portion of \eqref{eqn:Box-Box_tau} is bounded by $\tau^2\|\dd_t^3 u\|_{X^{s_1-1,b_1-1}}^2$.
	
	For $n\neq0$, the continuous derivative dominates. That is to say,
	\begin{equation*}
		\bigl|\bigl(\sg+\tfrac{2\pi}{\tau}n\bigr)^2+\bigl(\tfrac{e^{-i\tau\sg}-1}{\tau}\bigr)^2\bigr|\lesssim\bigl(\sg+\tfrac{2\pi}{\tau}n\bigr)^2.
	\end{equation*}
	Then for $(\sg,\xi)$ in the support of the integral,
	\begin{align*}
		&\Bigl|\sum_{n\neq0}\bigl(\bigl(\sg+\tfrac{2\pi}{\tau}n\bigr)^2+\bigl(\tfrac{e^{-i\tau\sg}-1}{\tau}\bigr)^2\bigr)\widetilde{u}(\sg+\tfrac{2\pi}{\tau}n,\xi)\Bigr|^2\\
		&\lesssim\Bigl|\sum_{n\neq0}\bigl(\sg+\tfrac{2\pi}{\tau}n\bigr)^2\widetilde{u}(\sg+\tfrac{2\pi}{\tau}n,\xi)\Bigr|^2\\
		&\lesssim \Bigl(\sum_{n\neq0}\bigl(\sg+\tfrac{2\pi}{\tau}n\bigr)^{-2(s+b-1)}\Bigr)\Bigl(\sum_{n\neq0}\bigl(\sg+\tfrac{2\pi}{\tau}n\bigr)^{2(s+b+1)}|\widetilde{u}(\sg+\tfrac{2\pi}{\tau}n,\xi)|^2\Bigr)\\
		&\lesssim\tau^{2(s+b-1)}\sum_{n\neq0}\lan|\sg+\tfrac{2\pi}{\tau}n|+|\xi|\ran^{2s}\lan|\sg+\tfrac{2\pi}{\tau}n|-|\xi|\ran^{2b}|\widetilde{\dd_tu}(\sg+\tfrac{2\pi}{\tau}n,\xi)|^2
	\end{align*}
	where we used that $|\sg+\tfrac{2\pi}{\tau}n|\geq2|\xi|$ in the last line.
	Crudely bounding $\lan\dt(|\sg|+|\xi|)\ran\lesssim\tau^{-1}$ and $\lan\dt(|\sg|-|\xi|)\ran^{2(b_1-1)}\lesssim1$, we obtain
	\begin{align*}	
		&\int_{-\pi/\tau}^{\pi/\tau}\int_{|\xi|\leq\pi/2\tau}\lan\dt(|\sg|+|\xi|)\ran^{2(s_1-1)}\lan\dt(|\sg|-|\xi|)\ran^{2(b_1-1)} \\
		&\quad\cdot\Bigl|\sum_{n\neq0}\bigl(\bigl(\sg+\tfrac{2\pi}{\tau}n\bigr)^2+\bigl(\tfrac{e^{-i\tau\sg}-1}{\tau}\bigr)^2\bigr)\widetilde{u}(\sg+\tfrac{2\pi}{\tau}n,\xi)\Bigr|^2 d\sg d\xi\\
		&\lesssim \tau^{-2(s_1-1)}\tau^{2(s+b-1)}\sum_{n\neq0}\int_{-\pi/\tau}^{\pi/\tau}\int_{|\xi|\leq\pi/2\tau} \lan|\sg+\tfrac{2\pi}{\tau}n|+|\xi|\ran^{2s}\lan|\sg+\tfrac{2\pi}{\tau}n|-|\xi|\ran^{2b}|\widetilde{\dd_tu}(\sg+\tfrac{2\pi}{\tau}n,\xi)|^2\,d\sg\,d\xi\\
		&\lesssim \tau^{2(s-s_1+b)}\|\dd_t u\|_{X^{s,b}}^2.\qedhere
	\end{align*}
\end{proof}
In Section \ref{sec:local-error} we will improve this lemma in the case that $u$ is a (frequency truncated) wave map.

\section{Bilinear estimates}\label{sec:bilinear-estimates}
In this section we will establish product estimates in the discrete Bourgain spaces. The results are stated below, however we postpone the proofs to the end of the section since they require some preparation. Henceforth fix $d=3$.

For $k\in \N=\{0,1,\ldots\}$, we define dyadic cut-offs 
\begin{equation}\label{eqn:vp_k}
\psi_k(\xi)=\chi(2^{-k}\xi)-\chi(2^{-(k-1)}\xi)\text{ for } k\neq0.
\end{equation}
and $\vp_0=\chi$.
Here $\chi$ is as in \eqref{eqn:chi}. We then define the (inhomogeneous) Littlewood-Paley multipliers $u_k:=P_k(u)=\F^{-1}(\psi_k\,\widehat{u})$ and $P_{\leq k}=\sum_{j=0}^kP_j$. Let also $\chi_{A}(t)$ denote the sharp indicator function to the set $A\subset\R$, not to be confused with the function $\chi$ introduced above.

Our first result concerns the continuity of (sharp) time cut-offs in spaces involving sufficiently few time derivatives. This will be used to handle the discontinuity in the definition of the iterates \eqref{SS}. Compare with \cite[Lemma 2.11]{tao2006nonlinear}.

For a Schwarz function $\eta(t)$, we let $\eta_T(t):=\eta(t/T)$.
\begin{lemma}\label{lem:ctnty-time-cut-off}
	Let $s,s_1\in\R$ and $-1/2<b<1/2$ such that $s+b<1/2$, $s_1+b>0$ and $s-1/2<s_1\leq s$. For any sequence of functions $(F_n)_n$ and Schwarz function $\eta$ the following hold uniformly in $0<T<1$:
	\begin{enumerate}
		\item\label{item:1}
		$
		\|\eta_T(n\tau)\, F_n\|_{X^{s_1,b}_\tau}\lesssim_{\eta,s,s_1,b} T^{s-s_1}\|F_n\|_{X^{s,b}_\tau},
		$
		\item\label{item:2}
		$
		\|\chi_{[0,T]}(n\tau)\, F_n\|_{X^{s_1,b}_\tau}\lesssim_{s,s_1,b} T^{s-s_1}\|F_n\|_{X^{s,b}_\tau}
		$
		\item\label{item:3}
		$
		\|\chi_{[0,\infty)}(n\tau)\, F_n\|_{X^{s,b}_\tau}\lesssim_{s,b} \|F_n\|_{X^{s,b}_\tau}.
		$
	\end{enumerate}
	The analogous estimates in the continuous Bourgain spaces are also all true.
\end{lemma}

The second result is the bilinear estimates which allow us to fully exploit the null structure in the wave maps nonlinearity.
\begin{theorem}[Bilinear estimates]\label{thm:bilinear_estimates}
	Let $s=3/2+s'$ and $b=1/2+b'$ such that $s'>b'>0$. Let $u=(u_n(x))_n$ and $v=(v_n(x))_n$ be sequences of functions with frequencies restricted to $|\xi|\leq\pi/4\tau$. Then for any $\ko,\kt\in \N$ and any $\eps>0$ it holds
	\begin{equation}
		\|P_{\ko}u_n\cdot P_{\kt}v_n\|_{\xsbt}\lesssim_{\eps,s,b} 2^{-(s'-b'-\eps)\min\{\ko,\kt\}}\|P_{\ko}u_n\|_{\xsbt}\|P_{\kt}v_n\|_{\xsbt}.\label{eqn:statement1}
	\end{equation}
	In particular, dyadic summation leads to
	\begin{equation}
		\|u_n\cdot v_n\|_{X^{s,b}_\tau}\lesssim_{s,b}\|u_n\|_{X^{s,b}_\tau}\|v_n\|_{X^{s,b}_\tau}.\label{eqn:algebra-property}
	\end{equation}
	If $F=(F_n(x))_n$ also has frequency restricted to $|\xi|\leq\pi/4\tau$ and $b'<1/2$, we also have the inhomogeneous estimate
	\begin{equation}
		\|P_{\ko}u_n\cdot P_{\kt}F_n\|_{X^{s-1,b-1}_\tau}\lesssim_{s,b}2^{-s'\min\{\kt,\kth\}}\|P_{\ko}u_n\|_{X^{s,b}_\tau}\|P_{\kt}F_n\|_{X^{s-1,b-1}_\tau}\label{eqn:statement2}
	\end{equation}
	which sums to
	\begin{equation}
		\|u_n\cdot F_n\|_{X^{s-1,b-1}_\tau}\lesssim_{s,b}\|u_n\|_{X^{s,b}_\tau}\|F_n\|_{X^{s-1,b-1}_\tau}.\label{eqn:xsbtm-property}
	\end{equation}
	The analogous estimates in the continuous Bourgain spaces are also all true, without the need to restrict frequencies.
\end{theorem}

We approach the proofs of the above results by dyadic decompositions in both frequency and modulation, exploiting geometric cancellations of the product in Fourier space. The modulation cut-off $Q_l$ is defined by
$$
Q_l u(n,x)=\tft^{-1}(\vp_l(|\sg|-|\xi|)\tft(u_n)(\sg,\xi))(n,x).
$$
Note that the semi-discrete Fourier transform of $Q_l u$ is given by $\vp_l(|\sg|-|\xi|)\tft(u)(\sg,\xi)$ for $-\pi/\tau\leq\sg\leq\pi/\tau$ and extends periodically outside that region.
Observe the following modulation Bernstein estimate.
\begin{lemma}\label{lem:modulation_bernstein}
	Let $2\leq p,q\leq\infty$. Then
	$$
	\|P_kQ_ju_n\|_{\ell^p_\tau L^q_x}\lesssim_{p,q}2^{3(\half-\f{1}{q})k}2^{(\half-\f{1}{p})j}\|P_kQ_j u_n\|_{\ell^2_\tau L^2_x}
	$$
\end{lemma}
\begin{proof}
	The gain in $2^k$ comes from the standard Bernstein estimate, so we may reduce to $q=2$. The case $p=2$ is trivial so it only remains to consider $p=\infty$. We compute
	\begin{align*}
		\|P_kQ_ju_n\|_{\ell^\infty_\tau L^2_x}
		&\simeq\sup_{n\in\Z}\biggl\|\int_{-\pi/\tau}^{\pi/\tau}e^{in\tau\sg}\vp_k(\xi)\vp_j(|\sg|-|\xi|)\tft(u_n)(\sg,\xi)\,d\sg\biggr\|_{L^2_\xi}\\
		&\lesssim\sup_{n\in\Z}\biggl\|2^{j/2}\biggl(\int_{-\pi/\tau}^{\pi/\tau}|\tft({P_kQ_j u_n})(\sg,\xi)|^2\,d\sg\biggr)^\half\biggr\|_{L^2_\xi}\\
		&\lesssim 2^{j/2}\|P_kQ_j u_n\|_{\ell^2_\tau L^2_x}.\qedhere
	\end{align*}
\end{proof}	

\vspace{1em}
\begin{center}
	$* \quad * \quad *$
\end{center}
\vspace{1em}

We now arrive at the main technical part of this work, which is central to studying modulation paraproducts in the discrete setting.
We first consider the interactions of a sequence of functions $w_n(x)$ with a sequence $\vp_n$ in the time variable only. To simplify notation, we will denote $w=(w_n(x))_n$, and similar for $\vp$, so that $\vp\cdot w=(\vp_n\cdot w_n(x))_n$. We aim to exhibit cancellations in the decomposition
$$
P_j\vp\cdot Q_rw=\sum_{l\geq0}Q_l(P_j\vp\cdot Q_rw).
$$
Here $P_j\vp$ refers to the dyadic cut-off in the $n$-variable, i.e. 
$$
\ft(P_j\vp)(\sg)=\psi_j(\sg)\ft(\vp)(\sg),\qquad(2^j\leq 4\pi/\tau,\ \sg\in[-\pi/\tau,\pi/\tau]).
$$
Recall
\begin{equation}\label{conv}
	\tft(\vp\cdot w)(\sg,\xi)=\f{1}{2\pi}\int_{\sg_1=-\pi/\tau}^{\pi/\tau}\ft\vp(\sg-\sg_2)\tft w(\sg_2,\xi)\,d\sg_2.
\end{equation}
In the proofs which follow, we let $\sg_1$ denote an arbitrary point in the support of $\ft\vp$ and $(\sg_2,\xi_2)$ a point in the support of $\tft w$. Then $(\sg,\xi)=(\sg_1+\sg_2,\xi_2)$ is an arbitrary point in the support of the product. By periodicity it suffices to consider $\sg,\,\sg_2\in[-\pi/\tau,\pi/\tau]$ and $\sg_1\in[-2\pi/\tau,2\pi/\tau]$. When $\sg_1\in[-\pi/\tau,\pi/\tau]$, the geometry reduces to that of the continuous case.

The cut-off $|\xi|\lesssim\tau^{-1}$ is not necessary for Lemmas \ref{lem:geom4}-\ref{claim4}.

\begin{lemma}\label{lem:geom4}
	For $l\in\Z$,
	$$
	Q_l(P_{\leq l-10}\vp\cdot Q_{\leq l-10}w)=0.
	$$
\end{lemma}
\begin{proof}
	Let $\sg,\sg_2\in[-\pi/\tau,\pi/\tau]$.
	Then in the classical case $\sg_1\in[-\pi/\tau,\pi/\tau]$, we estimate
	\begin{align*}
		||\sg_1+\sg_2|-|\xi_2||&\leq||\sg_2|-|\xi_2||+|\sg_1|\leq 2^{l-9}
	\end{align*}
	and the result is proved.
	
	If $\sg_1\in[\pi/\tau,2\pi/\tau]$, we have $|\sg_1-2\pi/\tau|\leq2^{l-10}$ and $0\leq\sg_1+\sg_2\leq \pi/\tau$. It follows that $||\sg_2|-\pi/\tau|\leq 2^{l-10}$ and so $||\xi_2|-\pi/\tau|\leq2^{l-10}$.
	On the other hand, since $\sg_2\geq-\pi/\tau$,
	$$
	\pi/\tau\geq \sg_1+\sg_2\geq\pi/\tau-2^{l-10}.
	$$
	Therefore $|\xi_2|,\,\sg_1+\sg_2\in[\pi/\tau-2^{l-10},\pi/\tau+2^{l-10}]$ and it follows that $||\sg_1+\sg_2|-|\xi_2||\leq 2^{l-7}$, and so the output to modulation $\sim 2^l$ vanishes.
\end{proof}

\begin{lemma}\label{lem:geom5}
	For $l\in\Z$, $r\geq l+10$
	$$
	Q_l(P_{\leq r-10}\vp\cdot Q_rw)=0.
	$$
\end{lemma} 
\begin{proof}
	Again assume $\sg,\sg_2\in[-\pi/\tau,\pi/\tau]$. If $\sg_1\in[-\pi/\tau,\pi/\tau]$, the result follows from the inequality $||\sg_1+\sg_2|-|\xi_2||\geq||\sg_2|-|\xi_2||-|\sg_1|$.
	
	If $\sg_1\in[\pi/\tau,2\pi/\tau]$, we use that $\sg_2\leq\pi/\tau-\sg_1\leq -\pi/\tau+2^{r-10}$ to see that
	$$
	|(\sg_1+\sg_2)-(-\sg_2)|\leq|\sg_1-2\pi/\tau|+2|\sg_2+\pi/\tau|\leq 2^{r-9}.
	$$
	Then
	$
	||\sg_1+\sg_2|-|\xi_2||\geq|(-\sg_2)-|\xi_2||-|(\sg_1+\sg_2)-(-\sg_2)|\geq 2^{r-3}
	$,
	which proves the result.
\end{proof}

Next we consider the product of spacetime sequences $u=(u_n(x))_n$ and $v=(v_n(x))_n$, via the decomposition
$$
P_{\ko}u\cdot P_{\kt}v=\sum_{l\geq0}Q_l(P_{\ko}u\cdot P_{\kt}v).
$$
Here the $P_{k}$ represent spatial frequency cut-offs.

We let $(\sg_1,\xi_1)$ and $(\sg_2,\xi_2)$ denote points in the Fourier support of $P_{\ko}u$ and $P_{\kt}v$ respectively, and $(\sg,\xi)=(\sg_1+\sg_2,\xi_1+\xi_2)$ an arbitrary point in the support of the product. Again it suffices to consider $\sg,\,\sg_1\in[-\pi/\tau,\pi/\tau]$ and $\sg_2\in[-2\pi/\tau,2\pi/\tau]$. It is helpful to denote
$$
S_k:=[(2k-1)\pi/\tau,(2k+1)\pi/\tau]\times\R^3.
$$

\begin{lemma}\label{claim2}
	For $l\geq\kt+10$, $\ko\geq\kt+10$, it holds
	$$
	Q_l (Q_{\leq l-10}P_{\ko}u\cdot Q_{\leq l-10}P_{\kt}v)=0.
	$$
\end{lemma}
\begin{proof}
	Let $(\sg,\xi),\ (\sg_1,\xi_1)\in S_0$. If also $(\sg_2,\xi_2)\in S_0$, it is easy to see that
	\begin{align*}
		||\sg_1+\sg_2|-|\xi_1+\xi_2||\leq||\sg_1|-|\xi_1||+||\sg_2|-|\xi_2||+2|\xi_2|\leq 2^{l-7}.
	\end{align*}
	It is thus impossible for the output to fall at modulation $2^{l-2}\leq ||\sg|-|\xi||\leq 2^l$.
	
	The case $(\sg_2,\xi_2)\in S_{+1}$ is demonstrated in Figure \ref{fig:modulation}.
	\begin{figure}[t]
		\includegraphics[scale=0.3]{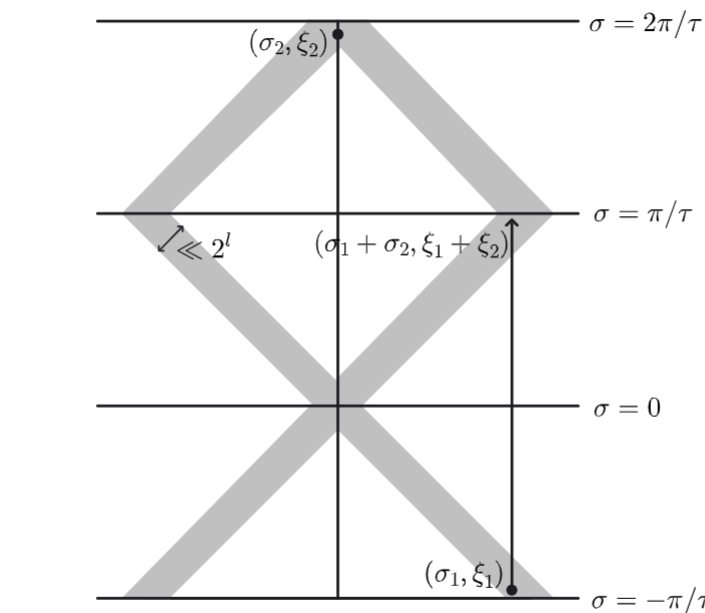}
		\caption{\label{fig:modulation}}
	\end{figure}
	Here the modulation cut-off on $v$ restricts to
	$$
	||\sg_2-2\pi/\tau|-|\xi_2||\leq2^{l-10}.
	$$
	Thus
	\begin{align}
		||\sg_1+\sg_2|-|\xi_1+\xi_2||&\leq||\sg_1+\sg_2|-\pi/\tau|+||\xi_1+\xi_2|-\pi/\tau|\nonumber\\
		&\leq ||\sg_2-2\pi/\tau|-|\xi_2||+|\sg_1+\pi/\tau|+||\xi_1|-\pi/\tau|+2|\xi_2|\nonumber\\
		&\leq ||\sg_2-2\pi/\tau|-|\xi_2||+2|\sg_1+\pi/\tau|+||\xi_1|-|\sg_1||+2|\xi_2|. \label{po9}
	\end{align}
	Then $\sg_1+\sg_2\leq\pi/\tau$ implies
	$$
	|\sg_1+\pi/\tau|\leq|2\pi/\tau-\sg_2|\leq ||2\pi/\tau-\sg_2|-|\xi_2||+|\xi_2|\leq 2^{l-9}.
	$$
	Substituting this into \eqref{po9}, we find $||\sg_1+\sg_2|-|\xi_1+\xi_2||\leq 2^{l-6}$, so it is impossible for the output to be at modulation $\geq 2^{l-2}$. The case $(\sg_2,\xi_2)\in S_{-1}$ can be treated similarly.
\end{proof}

\begin{lemma}\label{claim1}
	Let $\kt\leq\ko\leq\kt+10$, $l\geq\ko+10$. Then
	$$
	Q_l(Q_{\leq l-10}P_{\ko}u\cdot Q_{\leq l-10}P_{\kt}v)=0.
	$$
\end{lemma}
\begin{proof}
	As before we assume that $(\sg,\xi),\ (\sg_1,\xi_1)\in S_0$. If $(\sg_2,\xi_2)\in S_0$, the argument is as in the continuous case:
	\begin{align*}
		||\sg_1+\sg_2|-|\xi_1+\xi_2||\leq||\sg_2|-|\xi_2||+|\sg_1|+|\xi_1|	
		\leq 2^{l-8}
	\end{align*}
	showing that the output $(\sg,\xi)$ cannot be at large modulation.
	
	When $(\sg_2,\xi_2)\in S_{+1}$, we note that the restriction $\sg\in S_0$ enforces $2^l\leq8\pi\tau^{-1}$. Indeed, under the assumptions we have
	$$
	\pi/\tau\geq|\sg_1+\sg_2|\geq||\sg_1+\sg_2|-|\xi_1+\xi_2||-|\xi_1+\xi_2|\geq 2^{l-2}-2^{l-9}\geq 2^{l-3}.
	$$
	Then
	$$
	|\sg_2-2\pi/\tau|\leq||\sg_2-2\pi/\tau|-|\xi_2||+|\xi_2|\leq 2^{l-9}
	$$
	and similarly
	$$
	|\sg_1|\leq||\sg_1|-|\xi_1||+|\xi_1|\leq 2^{l-9}.
	$$
	It follows that
	\begin{align*}
		\sg_1+\sg_2\geq2\pi/\tau-(2\pi/\tau-\sg_2)-|\sg_1|\geq2\pi/\tau-2^{l-9}-2^{l-9}> \pi/\tau,
	\end{align*}
	contradicting $(\sg,\xi)\in S_0$. The case $(\sg_2,\xi_2)\in S_{-1}$ is similar.
\end{proof}

\begin{lemma}\label{claim3}
	Let $j\geq\max\{r,\ko,\kt\}+10$ and $|j-l|\geq10$. Then
	$$
	Q_r(Q_jP_{\ko}u\cdot Q_lP_{\kt}v)=0.
	$$
\end{lemma}
\begin{proof}
	By symmetry it suffices to consider $l\leq j-10$.
	As usual suppose that $(\sg,\xi),\ (\sg_1,\xi_1)\in S_0$. If $(\sg_2,\xi_2)\in S_0$ also, then $||\sg_1|-|\xi_1||\geq 2^{j-2}\geq2^{-(j-8)}|\xi_1|$. Therefore $|\sg_1|>|\xi_1|$ and
	\begin{align*}
		||\sg_1+\sg_2|-|\xi_1+\xi_2||\geq||\sg_1|-|\xi_1||-|\sg_2|-|\xi_2|
		\geq 2^{j-2}-2^{j-8}>2^{r+7}.
	\end{align*}
	The result follows.
	
	Now suppose $(\sg_2,\xi_2)\in S_{+1}$. Note that $|\sg_1|\geq|\xi_1|$ so
	$
	\pi/\tau\geq|\sg_1|\geq ||\sg_1|-|\xi_1||\geq2^{j-1}.
	$
	Since $(\sg_2,\xi_2)\in S_{+1}$, the modulation cut-off imposes $|\sg_2-2\pi/\tau|\leq 2^{l}+2^{\kt}\leq 2^{j-9}$. We deduce that
	\begin{align*}
		||\sg_1+\sg_2|-|\xi_1+\xi_2||&\geq |\sg_1+\sg_2-2\pi/\tau+2\pi/\tau|-|\xi_1|-|\xi_2|\\
		&\geq 2\pi/\tau-|\sg_1|-|\sg_2-2\pi/\tau|-2^{\ko}-2^{\kt}\\
		&\geq 2\pi/\tau-\pi/\tau-2^{j-9}-2^{j-9}\geq2^{r+8}
	\end{align*}
	which completes the proof.
\end{proof}

\begin{lemma}\label{claim4}
	Suppose $\kt\geq\ko+10$, $j\geq \ko+10$. Then for $l\leq j-10$ and any $r\leq j-10$,
	$$
	Q_l(Q_rP_{\ko}u\cdot Q_jP_{\kt}v)=0.
	$$
\end{lemma}
\begin{proof}
	We again start by assuming that all three pairs of Fourier variables belong to $S_0$. In this case the analysis is standard and we have
	\begin{align*}
		||\sg_1+\sg_2|-|\xi_1+\xi_2||\geq||\sg_2|-|\xi_2||-|\sg_1|-|\xi_1|\geq 2^{j-3},
	\end{align*}
	proving the result.
	
	Next we assume $(\sg_2,\xi_2)\in S_{+1}$, separately considering the cases $|\sg_2-2\pi/\tau|>|\xi_2|$ and $|\sg_2-2\pi/\tau|<|\xi_2|$. In the former case, since $\sg_2>\pi/\tau$, we have
	$$
	|\sg_2|-|\xi_2|\geq|\sg_2-2\pi/\tau|-|\xi_2|\geq2^{j-2}
	$$
	and the argument can be completed as in the case $(\sg_2,\xi_2)\in S_0$.
	
	If $|\sg_2-2\pi/\tau|<|\xi_2|$, first note that from $|\sg_1|\leq 2^{j-9}$ we have
	$$
	2\pi/\tau-\sg_2=2\pi/\tau-\sg+\sg_1\geq \pi/\tau-2^{j-9}.
	$$
	Then by the assumption $|\xi_2|-|2\pi/\tau-\sg_2|\geq 2^{j-2}$ we see that $|\xi_2|\geq\pi/\tau+2^{j-1}$, and so
	$$
	|\xi_1+\xi_2|\geq\pi/\tau+2^{j-1}-2^{\ko+1}\geq\pi/\tau+2^{j-2}.
	$$
	This yields
	$$
	||\xi_1+\xi_2|-|\sg_1+\sg_2||\geq(\pi/\tau+2^{j-2})-\pi/\tau= 2^{j-3}
	$$
	and the result follows.
\end{proof}

\subsection{Proof of Lemma \ref{lem:ctnty-time-cut-off}.}\label{subsec:ctny_time_cut_off}
We are now ready to prove Lemma \ref{lem:ctnty-time-cut-off}. We will only prove the discrete estimates, the continuous versions being similar. We will in fact prove the more general statement that whenever $\vp=(\vp_n)_n$ is a sequence satisfying
		\begin{equation}\label{eqn:condition}
		\|P_j\vp\|_{L^2_t}\lesssim_{s-s_1} (2^jT)^{s-s_1}2^{-j/2} \text{ for all  }j>0, \qquad \|P_{\leq k+10}\vp\|_{L^\infty_t}\lesssim_{s-s_1} (2^kT)^{s-s_1}
		\end{equation}
	for all $0<T<1$, $2^j,2^k\leq4\pi/\tau$ and some $s$, $s_1$ as in the statement of the lemma, it holds
	\begin{equation}\label{eqn:generalised_result}
		\|\vp (n\tau)\, F_n\|_{X^{s_1,b}_\tau}\lesssim_{s-s_1} T^{s-s_1}\|F_n\|_{\xsbt}.
	\end{equation}
	Note the conditions \eqref{eqn:condition} essentially require $\vp\in \dot{H}^{1/2-}$.
	
	We first check that the three functions under consideration indeed satisfy \eqref{eqn:condition}.
	\begin{enumerate}
		\item $\vp=\eta_T$: In this case $\|P_j\eta_T\|_{L^2_t}\lesssim (2^jT)^{s-s_1}2^{-j/2}$ holds for all $j\geq0$ and all $s-s_1\leq1$ thanks to the rapid decay of $\ft{\eta}$. The $L^\infty_t$ condition follows by Bernstein's inequality.
		\item $\vp=\chi_{[0,T]}$: Let $N=\lfloor T/\tau\rfloor$. An explicit computation shows that
		$$
		\ft{\chi}_{[0,T]}(\sg)=-2ie^{-i(N+1)\tau\sg/2}\f{\sin((N+1)\tau\sg /2)}{\dt(-\sg)}
		$$
		and it follows that \eqref{eqn:condition} is satisfied for any $0\leq s-s_1\leq 1$, again using Bernstein's inequality for the $L^\infty$ estimate.
		\item $\vp=\chi_{[0,\infty)}$:  The sharp cut-off has discrete Fourier transform
		$$
		\ft\chi_{[-,\infty)}(\sg)=\pi\dl_0-\text{p.v.}\Bigl(\f{1}{d_\tau(-\sg)}\Bigr)
		$$
		(extended periodically). The principal value is defined on $[-\pi/\tau,\pi/\tau]$ by
		$$
		\text{p.v.}\int_{-\pi/\tau}^{\pi/\tau}\frac{f(\sg)}{d_\tau(-\sg)}d\sg=\lim_{\dl\to0}\int_{\dl\leq|\sg|\leq\pi/\tau}\frac{f(\sg)}{d_\tau(-\sg)}d\sg.
		$$
		Thus $\|P_j\chi\|_{L^2_t}\lesssim 2^{-j/2}$ for all $j>0$, and since $\|\chi\|_{L^\infty_t}=1$ we deduce that \eqref{eqn:condition} is satisfied for $s-s_1=0$.
	\end{enumerate}
		
	\begin{proof}[Proof of \eqref{eqn:generalised_result}]
	Write
	$$
	\|\vp F\|_{X^{s_1,b}_\tau}^2\simeq\sum_{k,j\geq0}\|P_kQ_j(\vp F)\|_{X^{s_1,b}_\tau}^2
	$$
	and split into cases. 
	In an abuse of notation we will use $l\gg j$ to mean $l\geq j+10$ (or sometimes a larger constant if necessary), so that really $2^l\gg 2^j$. Similarly $l\sim j$ means $|l-j|<10$.
	\begin{enumerate}
		\item $k\ll j$ (in particular $j>0$): In this case
		$$
		\|P_kQ_j(\vp F)\|_{X^{s_1,b}_\tau}\lesssim 2^{(s_1+b)j}\|P_kQ_j(\vp F)\|_{L^2_tL^2_x}.
		$$
		First suppose $F=Q_{\ll j}F$. In this case we use Lemma \ref{lem:geom4} to restrict the frequencies of $\vp$ and use the Bernstein estimate Lemma \ref{lem:modulation_bernstein} to bound
		\begin{align*}
			2^{(s_1+b)j}\|P_kQ_j(\vp Q_{\ll j}F)\|_{L^2_tL^2_x}
			&\lesssim 2^{(s_1+b)j} \|P_{\gtrsim j}\vp\|_{L^2_t}\|P_kQ_{\ll j}F\|_{L^\infty_tL^2_x}\\
			&\lesssim 2^{(s_1+b)j}\bigl(\sum_{r\gtrsim j} (2^rT)^{s-s_1}2^{-r/2}\bigr)\bigl(\sum_{l\ll j}2^{l/2}\|P_kQ_lF\|_{L^2_tL^2_x}\bigr)\\
			&\lesssim T^{s-s_1}\sum_{l\ll j}2^{(s+b-1/2)(j-l)}\|P_kQ_lF\|_{\xsbt}.
		\end{align*}
		Thus by Cauchy-Schwarz we deduce
		$$
		\sum_{k\ll j}\|P_kQ_j(\vp Q_{\ll j}F)\|_{X^{s_1,b}_\tau}^2\lesssim T^{s-s_1}\|F\|_{X^{s,b}_\tau}
		$$
		since $s+b<1/2$.
		
		Next consider $F=Q_{\gtrsim j}F$. Here we use Lemma \ref{lem:geom5} to see that
		\begin{align*}
			2^{(s_1+b)j}\|P_kQ_j(\vp Q_{\gtrsim j}F)\|_{L^2_tL^2_x}&\lesssim 2^{(s_1+b)j}\sum_{l\gtrsim j}\|P_{\lesssim l}\vp\|_{L^\infty_t}\|P_kQ_lF\|_{L^2_tL^2_x}\\
			&\lesssim T^{s-s_1}\sum_{l\gtrsim j}2^{(s_1+b)(j-l)}\|P_kQ_lF\|_{\xsbt}
		\end{align*}
		which is again acceptable thanks to $s_1+b>0$.
		
		\item We now suppose $k\gtrsim j$, so that
		$$
		\|P_kQ_j(\vp F)\|_{X^{s_1,b}_\tau}\lesssim 2^{s_1k}2^{bj}\|P_kQ_j(\vp F)\|_{L^2_tL^2_x}.
		$$
		If $F=Q_{\gg j}F$, we apply Bernstein's inequality on the outer modulation and use Lemma \ref{lem:geom5} bound
		\begin{align*}
			2^{s_1k}2^{bj}\|P_kQ_j(\vp Q_{\gg j}F)\|_{L^2_tL^2_x}
			&\lesssim 2^{s_1k}2^{(b+1/2)j} \sum_{l\gg j} \|P_{\sim l}\vp\|_{L^2_t}\|P_kQ_lF\|_{L^2_tL^2_x}\\
			&\lesssim 2^{s_1k}2^{(b+1/2)j}\sum_{l\gg j} (2^lT)^{s-s_1}2^{-l/2}2^{-bl-s_1k-(s-s_1)l}\|P_kQ_lF\|_{\xsbt}\\
			&\lesssim T^{s-s_1}\sum_{l\gg j}2^{(b+1/2)(j-l)}\|P_kQ_lF\|_{\xsbt}
		\end{align*}
		which is again acceptable since $b>-1/2$.
		
		When $F=Q_{\ll j}F$, so in particular $j>0$, we use Lemmas \ref{lem:geom4} and \ref{lem:geom5} to estimate
		\begin{align*}
			2^{s_1k}2^{bj}\|P_kQ_j(\vp Q_{\ll j}F)\|_{L^2_tL^2_x}
			&\lesssim 2^{s_1k}2^{bj}\|P_{\sim j}\vp\|_{L^2_t}\sum_{l\ll j}\|P_kQ_lF\|_{L^\infty_t L^2_x}\\
			&\lesssim 2^{s_1k}2^{bj}(2^{j}T)^{s-s_1}2^{-j/2}\sum_{l\ll j}2^{l/2}2^{-sk}2^{-bl}\|P_kQ_lF\|_{\xsbt}\\
			&\lesssim T^{s-s_1}2^{(s-s_1)(j-k)}\sum_{l\ll j}2^{(b-1/2)(j-l)}\|P_kQ_lF\|_{\xsbt}
		\end{align*}
		which is acceptable since $b<1/2$.
		
		Lastly we study the case $F=Q_{\sim j}F$, where
		\begin{align*}
			2^{s_1k}2^{bj}\|P_kQ_j(\vp Q_{\sim j}F)\|_{L^2_tL^2_x}
			&\lesssim 2^{s_1k}2^{bj}\|P_{\lesssim j}\vp\|_{L^\infty_t}\|P_kQ_{\sim j}F\|_{L^2_tL^2_x}\\
			&\lesssim T^{s-s_1}2^{(s-s_1)(j-k)}\|P_kQ_{\sim j}F\|_{\xsbt}.
		\end{align*}
	\end{enumerate}
	This completes the proof of  \eqref{eqn:generalised_result} and so of Lemma \ref{lem:ctnty-time-cut-off}.
	\end{proof}

\subsection{Proof of Theorem \ref{thm:bilinear_estimates}.}
We will prove the discrete bilinear estimates following the argument laid out in \cite{marsden2024global} for the continuous case. We require one additional ingredient not necessary for Lemma \ref{lem:ctnty-time-cut-off}: the Strichartz estimates. These were proved for discrete wave equations by Ruff and Schnaubelt \cite{ruff} in dimension three. We use the notation
$$
\|f\|_{p,q}\equiv\|f_n(x)\|_{\ell^p_\tau L^q_x}= \Bigl(\tau\sum_n\|f_n(x)\|_{L^q_x}^p\Bigr)^{1/p}.
$$
\begin{theorem}[Semi-discrete Strichartz estimates for the wave operator, {\cite[Theorem 2.5]{ruff}}]\label{thm:strichartz}
	Let $2< p\leq\infty$, $2\leq q<\infty$ with $p^{-1}+q^{-1}\leq\half$. Set $\gamma=\f{3}{2}-\f{1}{p}-\f{3}{q}$.
	For $k\geq0$, $\tau\in(0,1)$ it holds
	\begin{align*}
		\|P_{k}e^{\pm in\tau|\na|}f\|_{\ell^p_\tau L^q_x}&\lesssim_{p,q}\|P_kf\|_{\dot{H}^\gamma_x}.
	\end{align*}
	Here $\F_x(e^{\pm in\tau|\na|}f)(\xi):=e^{in\tau|\xi|}\hat{f}(\xi)$. 
\end{theorem}
The next proposition rephrases the Strichartz estimates in the $X^{s,b}_\tau$ setting. Our proof closely follows that of \cite[Proposition 26]{burzio2020long}.

\begin{proposition}[Semi-discrete transferred Strichartz estimates.]\label{transfer_prop}
	Let $b>1/2$, $\tau\in(0,1)$ and $k\geq0$ with $2^k\leq\pi/2\tau$. Then for $p$, $q$ as in Theorem \ref{thm:strichartz}, $s=3/2+s'$, it holds
	\begin{equation}\label{eqn:MB}
	\|P_ku_n\|_{\ell^p_\tau L^q_x}\lesssim_{p,q,s,b}2^{-(\f{1}{p}+\f{3}{q}+s')k}\|P_ku_n\|_{X^{s,b}_\tau}.
	\end{equation}
	Moreover, for $m\in\Z\backslash\{0\}$,
	\begin{equation}\label{eqn:D_MB}
	\|P_k D_{m\tau}u_n\|_{\ell^p_\tau L^q_x}\lesssim_{p,q,s,b}2^{(1-\f{1}{p}-\f{3}{q}-s')k}\|P_k u_n\|_{\xsbt}.
	\end{equation}
\end{proposition}
In fact, by Bernstein's inequality we can also take $q=\infty$.
\begin{proof}
	We first consider \eqref{eqn:MB}. Write 
	\begin{align}
		P_ku_n(x)&=(2\pi)^{-4}\int_{\sg=-\pi/\tau}^{\pi/\tau}\int_{\xi\in\R^3}e^{i(x\cdot\xi+n\tau\sg)}\psi_k(\xi)\tft(u)(\sg,\xi)d\sg\,d\xi\nonumber\\
		&=(2\pi)^{-4}\iint e^{i(x\cdot\xi+n\tau\sg)}\psi_k(\xi)[\tft(u_{+})+\tft(u_{-})](\sg,\xi)d\sg\,d\xi\label{align2}
	\end{align}
	where
	$$
	\ft u_{\pm}(\sg,\xi):=\chi_{[0,\infty)}(\pm\sg)\tft(u)(\sg,\xi),
	$$
	extended periodically beyond $[-\pi/\tau,\pi/\tau]$. Make the change of variables $\sg=a\pm|\xi|$ to write
	\begin{align*}
		P_ku_n(x)&=(2\pi)^{-4}\int_{a=-\pi/\tau}^{\pi/\tau}e^{in\tau a}\int_{\xi} e^{i(x\cdot\xi+n\tau|\xi|)}\psi_k(\xi)\tft u_{+}(a+|\xi|,\xi)d\xi\,da\\
		&\quad+(2\pi)^{-4}\int_{a=-\pi/\tau}^{\pi/\tau}e^{in\tau a}\int_{\xi} e^{i(x\cdot\xi-n\tau|\xi|)}\psi_k(\xi)\tft u_{-}(a-|\xi|,\xi)d\xi\,da\\
		&=\f{1}{2\pi}\int_{a=-\pi/\tau}^{\pi/\tau}e^{in\tau a}[e^{in\tau|\na|}P_ku_{+}^{(a)}+e^{-in\tau|\na|}P_ku_{-}^{(a)}](x)da.
	\end{align*}
	Here
	$$
	\F_x u_{\pm}^{(a)}(\xi):=(\tft u_{\pm})(a\pm|\xi|,\xi).
	$$
	Now, Minkowski's inequality followed by the discrete Strichartz estimates yields
	\begin{align*}
		\|P_ku_n\|_{\ell^p_\tau L^q_x}&\lesssim\int_{a}\|P_ku_{+}^{(a)}\|_{\dot{H}^{\gamma}}+\|P_ku_{-}^{(a)}\|_{\dot{H}^\gamma}da\\
		&\lesssim\int_{a}\||\xi|^\gamma\psi_k(\xi)\tft u_{+}(a+|\xi|,\xi)\|_{L^2_\xi}+\||\xi|^\gamma\psi_k(\xi)\tft u_{-}(a-|\xi|,\xi)\|_{L^2_\xi}\,da.
	\end{align*}
	for $\gamma$ as in Theorem \ref{thm:strichartz}.
	Then
	\begin{align*}
		\int_{a=-\pi/\tau}^{\pi/\tau}\||\xi|^\gamma \psi_k(\xi)\tft u_+(a+|\xi|,\xi)\|_{L^2_\xi}\,da
		&\lesssim\sum_{j\geq0}\int_{a=-\pi/\tau}^{\pi/\tau}\psi_j(a)\||\xi|^\gamma\psi_k(\xi)\tft u_+(a+|\xi|,\xi)\|_{L^2_\xi}\,da\\
		&\lesssim\sum_{j\geq0}2^{j/2}\||\xi|^\gamma \psi_j(a)\psi_k(\xi)\tft u_+(a+|\xi|,\xi)\|_{L^2_{a,\xi,|a|\leq\pi/\tau}}\\
		&\lesssim\sum_{j\geq0}2^{j/2}\||\xi|^\gamma \psi_j(\sg-|\xi|)\psi_k(\xi)\tft u_+(\sg,\xi)\|_{L^2_{\sg,\xi,|\sg-|\xi||\leq\pi/\tau}}\\
		&\lesssim\sum_{j\geq0}2^{(\half-b)j}\|\lan\dt(\sg-|\xi|)\ran^b|\xi|^\gamma \psi_k(\xi)\tft u_+(\sg,\xi)\|_{L^2_{\sg,\xi,|\sg-|\xi||\leq\pi/\tau}}.
	\end{align*}
	By periodicity of the integrand in $\sigma$, this is equal to
	\begin{align*}
		\sum_{j\geq0}2^{(\half-b)j}\|\lan\dt(|\sg|-|\xi|)\ran^b|\xi|^\gamma \psi_k(\xi)\tft u_+(\sg,\xi)\|_{L^2_{\sg,\xi,|\sg|\leq\pi/\tau}}
	\end{align*}
	Since $b>1/2$ and $|\xi|\leq 2^k\leq\pi/2\tau$, this is bounded by $2^{(\gamma-s)k}\|P_ku\|_{X^{s,b}_\tau}$ as required. The term involving $u_-$ can be treated similarly.
	
	We now turn to \eqref{eqn:D_MB}. If the modulation of $u_n$ is at most comparable to $k$, the bound reduces immediately to the first statement using that $\tft(D_{m\tau}u_n)(\sg,\xi)=-d_{m\tau}(-\sg)\tft(u_n)(\sg,\xi)$.
	When the sequence $u_n$ is concentrated far from the lightcone the result follows from a direct application of Lemma \ref{lem:modulation_bernstein}.
\end{proof}

We can now prove Theorem \ref{thm:bilinear_estimates}. Having established Lemmas \ref{claim2}--\ref{claim4}, the proof is identical to the analysis in the continuous case. For this reason, we will only prove the first part of the theorem, \eqref{eqn:statement1}--\eqref{eqn:algebra-property}, using Lemmas \ref{claim2} and \ref{claim1}. The reader may consult for instance \cite[Appendix 2.C]{2024probabilistic} or \cite[Theorem 2.12]{geba2016introduction} for the proof of the second statement which requires Lemmas \ref{claim3} and \ref{claim4}.
\begin{proof}[Proof of \eqref{eqn:statement1}.]
	We may assume without loss of generality that $\ko\geq\kt$. We first consider the case $\ko\geq\kt+10$, so the whole term is at frequency $\sim 2^{\ko}$. First suppose the overall modulation is $\gg 2^{\ko}$, so we have to bound
	$$
	\Bigl\|\sum_{l\geq\ko+10}Q_lP_{\sim\ko}(P_{\ko}u\cdot P_{\kt}v)\Bigr\|_{\xsbt}.
	$$
	By Lemma \ref{claim2} we know that interactions of the type
	$$
	Q_l P_{\sim \ko}(Q_{\leq l-10}P_{\ko}u\cdot Q_{\leq l-10}P_{\kt}v)
	$$
	vanish, so at least one factor must be at relatively high modulation. First suppose that it is $P_{\ko}u$.
	Since $P_{\ko}u$ and $P_{\kt}v$ have frequencies truncated to $\pi/4\tau$, their product is at frequency at most $\pi/2\tau$. Thus $\lan\dt(|\sg|\pm|\xi|)\ran\simeq\lan|\sg|\pm|\xi|\ran$ we can estimate
	\begin{align*}
		\|Q_lP_{\sim\ko}(Q_{>l-10}P_{\ko}u\cdot P_{\kt}v)\|_{\xsbt}
		&\lesssim\sum_{j> l-10}2^{(s+b)l}\|Q_jP_{\ko}u\cdot P_{\kt}v\|_{\ell^2_\tau L^2_x}\\
		&\lesssim\sum_{j> l-10}2^{(s+b)l}\|Q_jP_{\ko}u\|_{\ell^2_\tau L^2_x}\|P_{\kt}v\|_{\ell^\infty_\tau L^\infty_x}\\
		&\lesssim2^{-s'\kt}\|P_{\kt}v\|_{\xsbt}\sum_{j>l-10}2^{(s+b)(l-j)}\|Q_jP_{\ko}u\|_{\xsbt}
	\end{align*}
	where we used the transferred Strichartz estimates to place $v$ into $X^{s,b}_\tau$, and directly placed $u$ into the Bourgain space using the frequency restrictions.
	Thus by the almost-orthogonality of the modulation cut-offs followed by the Cauchy-Schwarz inequality we obtain
	\begin{align*}
	\Bigl\|\sum_{l\geq\ko+10}Q_lP_{\sim\ko}(Q_{>l-10}P_{\ko}u\cdot P_{\kt}v)\Bigr\|_{\xsbt}^2&\lesssim\sum_{l\geq k+10}\|Q_lP_{\sim\ko}(Q_{>l-10}P_{\ko}u\cdot P_{\kt}v)\|_{\xsbt}^2\\
	&\lesssim2^{-s'\kt}\|P_{\ko}u\|_{\xsbt}\|P_{\kt}v\|_{\xsbt},
	\end{align*}
	which is better than required.
	When instead $P_{\kt}v$ is at high modulation, the same argument applies with the roles of $u$ and $v$ reversed.
	
	When the modulation is lower than in the previous case but still large relative to $\kt$, we again use Lemma \ref{claim2} to deduce that one of the two factors must be at comparable modulation, suppose it is $P_{\kt}v$. Then for $\kt+10\leq l\leq \ko+10$ we estimate
	\begin{align*}
	\|P_{\sim\ko}Q_l(P_{\ko}u\cdot Q_{> l-10}P_{\kt}v)\|_{\xsbt}&\lesssim\sum_{j> l-10}2^{s\ko}2^{bl}\|P_{\ko}u\|_{\infty,2}\|Q_{j}P_{\kt}v\|_{2,\infty}\\
	&\lesssim \sum_{j>l-10}2^{b(l-j)}2^{-s'\kt}\|P_{\ko}u\|_{\xsbt}\|Q_jP_{\kt}v\|_{\xsbt}
	\end{align*}
	where we use the Strichartz estimate to bound $\|P_{\ko}u\|_{\infty,2}\lesssim 2^{-s\ko}\|P_{\ko}u\|_{\xsbt}$, and Bernstein's inequality to estimate $\|Q_jP_{\kt}v\|_{2,\infty}\lesssim 2^{3\kt/2}\|Q_jP_{\kt}v\|_{2,2}$ and place $v$ into $\xsbt$. We thus obtain
	\begin{align*}
	\|P_{\sim\ko}Q_{\kt\ll\cdot\lesssim\ko}(P_{\ko}u\cdot Q_{> l-10}P_{\kt}v)\|_{\xsbt}\lesssim2^{-s'\kt}\|P_{\ko}u\|_{\xsbt}\|P_{\kt}v\|_{\xsbt}.
	\end{align*}
	
	A similar argument applies when instead $P_{\ko}u$ is at high modulation, placing $P_{\ko}u$ into $\ell^2_\tau L^2_x$ and $P_{\kt}v$ into $\ell^\infty_\tau L^\infty_x$.
	
	When the output modulation is small, $\leq 2^{\kt+10}$, there are no cancellations to consider but we must separately consider the cases where $P_{\ko}u$ appears at low or high modulation relative to the output. Let $l<\kt+10$. When $u$ appears at high modulation we estimate
	\begin{align*}
		\|P_{\sim \ko}Q_l(Q_{\gtrsim l}P_{\ko}u\cdot P_{\kt}v)\|_{\xsbt}&\lesssim 2^{s\ko}2^{bl}\|Q_{\gtrsim l}P_{\ko}u\|_{2,2}\|P_{\kt}v\|_{\infty,\infty}\\
		&\lesssim 2^{-s'\kt}\|P_{\kt}v\|_{\xsbt}\sum_{j\gtrsim l}2^{b(l-j)}\|Q_jP_{\ko}u\|_{\xsbt}
	\end{align*}
	which is acceptable. 
	
	When $u$ appears at low modulation, let $M>2$ and estimate
	\begin{align*}
		\|P_{\sim \ko}Q_l(Q_{\ll l}P_{\ko}u\cdot P_{\kt}v)\|_{\xsbt}&\lesssim\sum_{j\ll l}2^{s\ko}2^{bl}\|Q_jP_{\ko}u\|_{M,2}\|P_{\kt}v\|_{\f{2M}{M-2},\infty}\\
		&\lesssim 2^{b(l-\kt)}2^{-(s'-b'-1/M)\kt}\|P_{\ko}u\|_{\xsbt}\|P_{\kt}v\|_{\xsbt}.
	\end{align*}
	This gives the required upper bound upon setting $\eps=1/M$ and square-summing over $l<\kt+10$.
	
	We now turn to the case $\kt\leq\ko\leq\kt+10$. Note in this case the whole term is at frequency at most $2^{\ko+1}$. If the output modulation is at most comparable to $2^{\ko}$, we again take $M>2$ and estimate
	\begin{align*}
		\|P_{\lesssim\ko}Q_{\lesssim\ko}(P_{\ko}u\cdot P_{\kt}v)\|_{\xsbt}
		&\lesssim2^{(s+b)\ko}\|P_{\ko}u\|_{\f{2M}{M-2},\f{1}{\eps}}\|P_{\kt}v\|_{M,\f{2M}{M-2}}\\
		&\lesssim 2^{-(s'-b')\ko}\|P_{\ko}u\|_{\xsbt}\|P_{\kt}v\|_{\xsbt}
	\end{align*}
	using that $\ko\simeq\kt$.
	
	The final case to consider is when the output modulation $2^l$ is much larger than $2^{\ko}$. In this case Lemma \ref{claim1} shows that one of $u$ or $v$ must also be at high modulation. By symmetry we may suppose it is $u$. Then for $l\geq \ko+10$ we estimate
	\begin{align*}
		\|P_{\lesssim\ko}Q_{l}(Q_{\gtrsim l}P_{\ko}u\cdot P_{\kt}v)\|_{\xsbt}
		&\lesssim2^{(s+b)l}\|Q_{\gtrsim l}P_{\ko}u\|_{2,2}\|P_{\kt}v\|_{\infty,\infty}\\
		&\lesssim 2^{-s'\kt}\|P_{\kt}v\|_{\xsbt}\sum_{j\gtrsim l}2^{(s+b)(l-j)}\|Q_jP_{\ko}u\|_{\xsbt}
	\end{align*} 
	which is acceptable once square-summing $l\geq \ko+10$.
\end{proof}

To complete the proof of the first part of Theorem \ref{thm:bilinear_estimates}, i.e. \eqref{eqn:algebra-property}, we write
\begin{align*}
\|u\cdot v\|_{\xsbt}\lesssim\sum_{k\geq0}\|P_k(u\cdot v)\|_{\xsbt}^2
&\lesssim \sum_{k\geq0}\|P_k(P_{\ll k}u\cdot P_{\sim k}v)\|_{\xsbt}^2+\sum_{k\geq0}\|P_k(P_{\sim k}u\cdot P_{\ll k}v)\|_{\xsbt}^2\\
&\quad+\sum_{k\geq0}\|P_k(P_{\gtrsim k}u\cdot P_{\gtrsim k}v)\|_{\xsbt}^2
\end{align*}
where 
$$
P_k(P_{\gtrsim k}u\cdot P_{\gtrsim k}v)=\sum_{j\gtrsim k}P_k(P_j u\cdot P_{\sim j}v).
$$
The estimate then follows from \eqref{eqn:statement1} and liberal application of the Cauchy-Schwarz inequality.

	
\section{Trilinear estimates with gain in time}\label{sec:trilinear-estimates}
The goal of this section is to prove trilinear estimates which improve upon Theorem \ref{thm:bilinear_estimates} by a gain in time over small intervals, necessary for the bootstrap argument in Section \ref{section:main_theorem}. We will follow the argument laid out in \cite{marsden2024global}. 

First observe that multiplication by a smooth cut-off $\eta_T$ with rapidly decaying Fourier transform does not affect the geometry of the interactions in a serious way. This is expressed through the estimate
$$
\|P_l \eta_T(n\tau)\|_{\ell^p_\tau}\lesssim_{N,p}T^{1/p}(2^lT)^{-N}
$$
for $2\leq p\leq\infty$, $l\geq0$ and $N>0$. The case $p=2$ was already discussed in Section \ref{subsec:ctny_time_cut_off}, and the general case is straightforward via the Haussdorff--Young inequality.

The main result of this section is the following trilinear estimate which includes a gain over small time intervals. The upper limits on $s$ and $b$ ensure that no more than $1/2$ time derivatives ever hit the cut-off $\eta_T$, which can thus be controlled as in Lemma \ref{lem:ctnty-time-cut-off}. Note however that unlike in that lemma, the estimates below see no loss of regularity. This is thanks to the specific nonlinear structure of the expression being estimated.

Fix $p\in\R^3$. For a function $f\in p+X^{s,b}_\tau$, we abusively use $\|f\|_{\xsbt}$ to denote $\|f-p\|_{\xsbt}$.
\begin{theorem}\label{thm:trilinear_thm}
Let $s=3/2+s'$, $b=1/2+b'$ with $0<b'<s'$ and $s'+b'<1/2$.
Let $f,\,g,\,h\in p+X^{s,b}_{\tau}$, each with frequency supported in $\{|\xi|\leq \pi/(8\tau^{1/2})\}$, and let $\eta$ be a Schwarz function.
Denote
$$
\T(f,g,h)_n=-f_n(\Box_\tau(g\cdot h)_n-g_n\,\Box_\tau h_n-h_n\,\Box_\tau g_n)
$$
and let $T>0$. 
Then
\begin{equation}
\|\eta_T(n\tau) \T(f,g,h)_n\|_{\xsbtm}\lesssim_{\eta,s,b} T^{\alpha(s,b)} (1+\|f\|_{X^{s,b}_{\tau}})\|g\|_{X^{s,b}_{\tau}}\|h\|_{X^{s,b}_{\tau}}.\label{eqn:trilinear}
\end{equation}
where $\alpha(s,b)=\min\{b',(s'-b')/4,(1/2-(s'+b'))/2\}$.

If in fact $g$ or $h\in X^{s,b}$ directly (as opposed to $p+\xsbt$), the estimate still holds. When $f\in \xsbt$ we can replace $(1+\|f\|_{\xsbt})$ with $\|f\|_{\xsbt}$.

The analogous continuous estimates are also true.
\end{theorem}
\begin{proof}
Our first step is to perform a frequency decomposition on $\T_\tau(f,g,h)$. To reduce notation we will write $u_k$ for $P_ku$, and similar for $P_{\lesssim k}u$ etc.. Observe that
\begin{align}
\T(f,g,h)=\sum_{k\geq0}P_k \T_\tau(f,g,h)
&=\sum_{k\geq0}\sum_{\ko\geq0}P_k\T(f_{\ko},g_{\lesssim\ko},h_{\lesssim \ko})\label{eqn:HLL}\\
&\quad+\sum_{k\geq0}\sum_{\kt\geq0}P_k\T(f_{\ll \kt},g_{\kt},h_{\leq\kt})\label{eqn:LHH1}\\
&\quad+\sum_{k\geq0}\sum_{\kth\geq0}P_k\T(f_{\ll\kth},g_{\leq\kth},h_{\kth}).\label{eqn:LHH2}
\end{align}
The second and third lines above can be treated identically, so it suffices to study \eqref{eqn:HLL} and \eqref{eqn:LHH1}.

We start with \eqref{eqn:HLL}. We must bound
$$
\sum_{k\geq0}\Bigl(\sum_{\ko\geq0}\|P_k(\eta_T\, \T(f_{\ko},g_{\lesssim\ko},h_{\lesssim \ko}))\|_{\xsbtm}\Bigr)^2.
$$

This is one of the most straightforward cases since the derivatives appear at low frequency. Here the null structure is not helpful, so we reverse it via the identity
\begin{align}
	\T(f,g,h)_n
	&=f_n[2(D_\tau g_n\cdot D_\tau h_{n-1}-\na_x g_n\cdot\na_x h_n)-2 D_\tau g_n\cdot D_{2\tau}h_n+2D_\tau g_{n-1}D_{2\tau}h_n]\nonumber\\
	&\simeq f Dg\cdot Dh,\label{eqn:heuristic}
\end{align}
where $D$ could represent any of the derivatives $\na_x$ or $D_{\tau}$, $D_{2\tau}$, and we neglect the subscripts on $f$, $g$ and $h$ which may be shifted by an index.

First suppose that the term outputs to low modulation $\leq 2^{\ko+30}$. Noting that the highest frequency $\ko$ must be at least comparable to $k$, we have
\begin{align*}
&\sum_{k\geq0}\Bigl(\sum_{\ko\geq0}\|P_kQ_{\lesssim\ko}(\eta_T\, \T(f_{\ko},g_{\lesssim\ko},h_{\lesssim \ko}))\|_{\xsbtm}\Bigr)^2\\
&\lesssim \sum_{k\geq0}\Bigl(\sum_{\ko\gtrsim k}\sum_{l\lesssim\ko}2^{(s-1)\ko}2^{(b-1)l}2^{l/2}\|P_kQ_{\lesssim\ko}(\eta_T\, \T(f_{\ko},g_{\lesssim\ko},h_{\lesssim \ko}))\|_{\ell^1_\tau L^2_x}\Bigr)^2
\end{align*}
by the modulation Bernstein estimate.

Let $1/M:=b'$. Then, first assuming $\ko\neq0$ so that $f_{\ko}=(f-p)_{\ko}$ (while trivially $Dg_{\ll k}=D(g-p)_{\ll k}$, $Dh_{\ll k}=D(h-p)_{\ll k}$), we use the modulation Bernstein estimate on scale $2^l$ followed by the transferred Strichartz estimate \eqref{eqn:D_MB} to estimate
\begin{align*}
&\sum_{k\geq0}\Bigl(\sum_{\ko>0}\|P_kQ_{\lesssim\ko}(\eta_T\, \T(f_{\ko},g_{\lesssim\ko},h_{\lesssim \ko}))\|_{\xsbtm}\Bigr)^2\\
&\lesssim \sum_{k\geq0}\Bigl(\sum_{\substack{\ko\gtrsim k,\\\ko\neq0}}\sum_{l\lesssim\ko}2^{(s-1)\ko}2^{(b-1)l}2^{l/2}\|\eta_T\|_{M}\|f_{\ko}\|_{\infty,2}\|Dg_{\lesssim\ko}\|_{\f{2M}{M-1},\infty}\|Dh_{\lesssim\ko}\|_{\f{2M}{M-1},\infty}\Bigr)^2\\
&\lesssim (T^{1/M}\|f\|_{\xsbt}\|g\|_{\xsbt}\|h\|_{\xsbt})^2\sum_{k\geq0}\Bigl(\sum_{\substack{\ko\gtrsim k,\\\ko\neq0}}\sum_{l\lesssim\ko}2^{(s-1)\ko}2^{b'l}2^{-s\ko}2^{2(\half+\f{1}{2M}-s')\ko}\Bigr)^2\\
&\lesssim (T^{b'}\|f\|_{\xsbt}\|g\|_{\xsbt}\|h\|_{\xsbt})^2
\end{align*}
which is as required.

Since $f\in p+X^{s,b}$, we cannot place $f_{\ko}$ into $\ell^\infty_\tau L^2_x$ if $\ko=0$. However this is a small issue to overcome since such a case can only arise for $k\leq20$. Taking again $1/M=b'$, we then have
\begin{align*}
&\sum_{k\leq 20}\|P_kQ_{\leq30}(\eta_T\, \T(f_{0},g_{\leq10},h_{\leq10}))\|_{\xsbtm}^2\\
&\lesssim\sum_{k\leq20}(2^{2(s-1)k}\|\eta_T\|_{M}\|f_0\|_{\infty,\infty}\|Dg_{\leq10}\|_{\f{2M}{M-2},\infty}\|Dh_{\leq10}\|_{\infty,2})^2\\
&\lesssim (T^{b'}(1+\|f\|_{\xsbt})\|g\|_{\xsbt}\|h\|_{\xsbt})^2
\end{align*}
where we used that $\|f_0\|_{\infty,\infty}\leq|p|+\|f\|_{X^{s,b}_{\tau}}$ (with the understanding that $\|f\|_{\xsbt}$ means $\|f-p\|_{\xsbt}$) and $\|Dh\|_{\infty,2}\lesssim\|Dh_r\|_{\xsbt}$.
Henceforth we will not comment on the exceptions at $\ko=0$.

Now suppose the term is at large modulation, so we are studying
\begin{align}
&\sum_{k\geq0}\Bigl(\sum_{\ko\gtrsim k}\|P_kQ_{\gg\ko}(\eta_T\, \T(f_{\ko},g_{\lesssim\ko},h_{\lesssim \ko}))\|_{\xsbtm}\Bigr)^2\nonumber\\
&\leq \sum_{k\geq0}\Bigl(\sum_{\ko\gtrsim k}\sum_{l\gg\ko}\|P_kQ_l(\eta_T\, \T(f_{\ko},g_{\lesssim\ko},h_{\lesssim \ko}))\|_{\xsbtm}\Bigr)^2\label{eqn:high-mod}
\end{align}
By Lemmas \ref{claim2} and \ref{claim1} we have $Q_l(Q_{\leq l-10}u_r\cdot Q_{\leq l-10}v_s)=0$ whenever $l\geq\max\{r,s\}+10$. Iterating this we find that
$$
Q_l(Q_{\leq l-10}u_{\ko} \cdot Q_{\leq l-20}v_{\lesssim \ko}\cdot Q_{\leq l-20} w_{\lesssim \ko})=0
$$
for $l\geq \ko+30$.
Then by Lemma \ref{lem:geom4}, $Q_{>l-10}(P_{\leq l-20}\eta_T \cdot Q_{\leq l-20}f_{\sim \ko})=0$, so in fact
$$
Q_l(P^{(t)}_{\leq l-20}\eta_T \cdot Q_{\leq l-20}f_{\ko} \cdot Q_{\leq l-20}Dg_{\lesssim \ko}\cdot Q_{\leq l-20} Dh_{\lesssim \ko})=0.
$$
Therefore in \eqref{eqn:high-mod} at least one of the 4 factors must be at modulation (or frequency in the case of $\eta_T$) at least $2^{l-20}$. 

First suppose $\eta_T$ is at high frequency. If $s'\geq1/6$, define $N=(1/2+s'+b')/2$. If $s'<1/6$, set $N=(1/2(1+b'-2s')$. In both cases $N>s'+b'$. Then 
\begin{align*}
&\|P_kQ_l(P_{\gtrsim l}\eta_T\, \T(f_{\ko},g_{\lesssim\ko},h_{\lesssim \ko}))\|_{\xsbtm}\\
&2^{(s-1+b-1)l}\|P_{\gtrsim l}\eta_T\|_2\|f_{\ko}\|_{\infty,2}\|Dg_{\lesssim \ko}\|_{\infty,\infty}\|Dh_{\lesssim \ko}\|_{\infty,\infty}\\
&\lesssim 2^{(s'+b')l}T^{1/2}(2^lT)^{-N}2^{-s\ko}2^{2(1-s')\ko}\|f\|_{\xsbt}\|g\|_{\xsbt}\|h\|_{\xsbt}\\
&\lesssim T^{\min\{s'-b'/2,(1/2-(s'+b'))/2\}}2^{(s'+b'-N)l}2^{(1/2-3s')\ko}\|f\|_{\xsbt}\|g\|_{\xsbt}\|h\|_{\xsbt}
\end{align*}
which is acceptable when summed as in \eqref{eqn:high-mod}.

If rather $f_{\ko}$ appears at high modulation, set $1/M=s'$. We place the high modulation factor into $\ell^\f{2M}{M-2}_\tau L^2_x$ and then directly into $\xsbt$ via the modulation Bernstein estimate to bound
\begin{align*}
\|P_kQ_l(\eta_T\, Q_{\gtrsim l}f_{\ko}\, Dg_{\lesssim\ko}\, Dh_{\lesssim \ko}))\|_{\xsbtm}
&\lesssim2^{(s'+b')l}\|\eta_T\|_{M}\|Q_{\gtrsim l}f_{\ko}\|_{\f{2M}{M-2},2}\|Dg_{\lesssim \ko}\|_{\infty,\infty}\|Dh_{\lesssim\ko}\|_{\infty,\infty}\\
&\lesssim T^{s'}2^{-(2-s')l}2^{2(1-s')\ko}\|f\|_{\xsbt}\|g\|_{\xsbt}\|h\|_{\xsbt}
\end{align*}
which is acceptable.

When $g_{\lesssim\ko}$ is at high modulation, we take $1/M=s'+b'$ and have
\begin{align*}
	\|P_kQ_l(\eta_T\, f_{\ko}\, Q_{\gtrsim l}Dg_{\lesssim\ko}\, Dh_{\lesssim \ko}))\|_{\xsbtm}
	&\lesssim T^{s'+b'}\|f_{\ko}\|_{\infty,\infty}\|Q_{\gtrsim l}Dg_{\lesssim\ko}\|_{\f{2M}{M-2},2}\|Dh_{\lesssim\ko}\|_{\infty,\infty}\\
	&\lesssim T^{s'+b'}2^{-l}2^{(1-2s')\ko}(1+\|f\|_{\xsbt})\|g\|_{\xsbt}\|h\|_{\xsbt}
\end{align*}
The case where $h_{\lesssim\ko}$ is at high modulation is analogous.

This completes the study of \eqref{eqn:HLL} and we now turn to \eqref{eqn:LHH1}. 
We start with the more delicate case in which the output falls at low modulation $\lesssim 2^{\kt}$,
\begin{equation}\label{eqn:sum}
\sum_{k\geq0}\Bigl(\sum_{\kt\gtrsim k}\|P_kQ_{\lesssim\kt}(\eta_T\,\T(f_{\ll \kt},g_{\kt},h_{\leq \kt}))\|_{\xsbtm}\Bigr)^2.
\end{equation}
The most interesting interactions occur when $h=h_{\leq \kt}$ is at high frequency relative to $T^{-1}$. Here Strichartz arguments are insufficient and we must appeal to the null structure to replace the high frequency derivatives with low frequency derivatives along the lightcone. Re-writing $\T(f,g,h)$ in terms of the wave operator, we first use Lemma \ref{lem:ctnty-time-cut-off} to discard of the time cut-off:
\begin{align*}
	&\sum_{k\geq0}\Bigl(\sum_{\kt\gtrsim k}\|P_kQ_{\lesssim\kt}(\eta_T\,\T(f_{\ll \kt},g_{\kt},h_{-\log_2T\leq\cdot\leq \kt}))\|_{\xsbtm}\Bigr)^2\\
	&\lesssim\sum_{k\geq0}\Bigl(\sum_{\kt\gtrsim k}\sum_{-\log_2T\leq\kth\leq \kt}\|\T(f_{\ll \kt},g_{\kt},h_{\kth})\|_{\xsbtm}\Bigr)^2.
\end{align*}
We then apply the bilinear estimates of Theorem \ref{thm:bilinear_estimates} followed by Lemma \ref{lem4} with $\eps=(s'-b')/2$ to estimate
\begin{align*}
	\|\T(f_{\ll \kt},g_{\kt},h_{\kth})\|_{\xsbtm}
	&=\|f_{\ll \kt}\left(\Box_\tau(g_{\kt}\cdot h_{\kth})-g_{\kt}\cdot\Box_\tau h_{\kth}-h_{\kth}\cdot \Box_\tau g_{\kt}\right)\|_{\xsbtm}\\
	&\lesssim (1+\|f_{\ll \kt}\|_{\xsbt})\|\Box_\tau(g_{\kt}\cdot h_{\kth})-g_{\kt}\cdot\Box_\tau h_{\kth}-h_{\kth}\cdot \Box_\tau g_{\kt}\|_{\xsbtm}\\
	&\lesssim 2^{-(s'-b')\kth/2}(1+\|f\|_{\xsbt})\|g_{\kt}\|_{\xsbt}\|h_{\kth}\|_{\xsbt}.
\end{align*}
When $\kt\simeq k$ we directly estimate $2^{-(s'-b')\kth/2}\leq T^{(s'-b')/2}$ and the desired bound follows upon summing the estimate.
On the other hand when $\kt\gg k$ we must have that $\kth \simeq \kt$ and we bound
$$
2^{-(s'-b')\kth/2}\lesssim T^{(s'-b')/4}2^{-(s'-b')\kt/4},
$$
which again gives the desired result upon summation.

When $h$ falls rather at frequency $2^{\kth}\leq T^{-1}$, i.e. $h=h_{\leq \min\{-\log_2T,\kt\}}$, we set $\f{1}{M}=\f{s'+3b'}{4}\in(b',1/2)$ and estimate
\begin{align*}
&\|P_kQ_{\lesssim\kt}(\eta_T\,\T(f_{\ll \kt},g_{\kt},h_{\leq \min\{-\log_2T,\kt\}}))\|_{\xsbtm}\\
&\lesssim \sum_{l\lesssim\kt}2^{(s-1)\kt}2^{(b-1)l}2^{(\half-\f{1}{M})l}\|P_kQ_{\lesssim\kt}(\eta_T\,\T(f_{\ll \kt},g_{\kt},h_{\leq \min\{-\log_2T,\kt\}}))\|_{\ell^{\f{M}{M-1}}_\tau L^2_x}\\
&\lesssim \sum_{l\geq0}2^{(s-1)\kt}2^{(b'-\f{1}{M})l}\|\eta_T\|_{\f{2M}{M-1}}\|f_{\ll\kt}\|_{\infty,\infty}\|Dg_{\kt}\|_{\infty,2}\|Dh_{\leq \min\{-\log_2T,\kt\}}\|_{\f{2M}{M-1},\infty}\\
&\lesssim\sum_{\kth\lesssim -\log_2T} T^{\half-\f{1}{2M}}2^{(\half+\f{1}{2M}-s')\kth}(1+\|f\|_{\xsbt})\|g_{\kt}\|_{\xsbt}\|h_{\kth}\|_{\xsbt}.
\end{align*}
Now, if $\kt\simeq k$, we estimate
\begin{align*}
\sum_{\kth\lesssim -\log_2T} T^{\half-\f{1}{2M}}2^{(\half+\f{1}{2M}-s')\kth}(1+\|f\|_{\xsbt})\|g_{\kt}\|_{\xsbt}\|h_{\kth}\|_{\xsbt}
&\lesssim T^{3(s'-b')/4}(1+\|f\|_{\xsbt})\|g_{\kt}\|_{\xsbt}\|h\|_{\xsbt}
\end{align*}
which is acceptable upon summing as in \eqref{eqn:sum}.

If $\kt\gg k$, then $h$ must fall at a comparable frequency and we estimate
\begin{align*}
&\sum_{\substack{\kth\lesssim-\log_2T,\\\kth\sim\kt}}T^{\half-\f{1}{2M}}2^{(\half+\f{1}{2M}-\half(s'+b'))\kth}2^{-\half(s'-b')\kth}(1+\|f\|_{\xsbt})\|g_{\kt}\|_{\xsbt}\|h_{\kth}\|_{\xsbt}\\
&\lesssim T^{\half(s'+b')-\f{1}{M}}2^{-\f{1}{2}(s'-b')\kt}(1+\|f\|_{\xsbt})\|g\|_{\xsbt}\|h\|_{\xsbt}\\
&\lesssim T^{(s'-b')/4}(1+\|f\|_{\xsbt})\|g\|_{\xsbt}\|h\|_{\xsbt}
\end{align*}
as required. 
This concludes the study of \eqref{eqn:LHH2} with a low modulation output. 

The final case to consider is \eqref{eqn:LHH2} with a high modulation output,
\begin{equation}\label{eqn:sum2}
	\sum_{k\geq0}\Bigl(\sum_{\kt\gtrsim k}\sum_{l\gg\kt}\|P_kQ_l(\eta_T\,\T(f_{\ll \kt},g_{\kt},h_{\leq \kt}))\|_{\xsbtm}\Bigr)^2.
\end{equation}
As in the study of \eqref{eqn:LHH1}, we see that one of the four factors must be at comparably high modulation.

When $\eta_T$ is at high modulation, we consider two cases. If $s'>1/4$, we set $N=1/4+(s'+b')/2$. If $s'\leq1/4$, set $N=1/2-(s'-b')/4$. In either case, $N>s'+b'$. Then
\begin{align*}
	&
	\|P_kQ_l(P_{\gtrsim l}\eta_T\T(f_{\ll \kt},g_{\kt},h_{\leq\kt})\|_{\xsobot}\\
	&\lesssim 2^{(s'+b')l}\|P_{\gtrsim l}\eta_T\|_2\|f_{\ll\kt}\|_{\infty,\infty}\|Dg_{\kt}\|_{\infty,2}\|Dh_{\leq\kt}\|_{\infty,\infty}\\
	&\lesssim2^{(s'+b')l}T^{1/2}(2^lT)^{-N}2^{(1-s)\kt}2^{(1-s')\kt}(1+\|f\|_{\xsbt})\|g\|_{\xsbt}\|h\|_{\xsbt}\\
	&\lesssim (T^{\min\{(s'-b')/4,(1/2-(s'+b'))/2\}}2^{(s'+b'-N)l}2^{(1/2-2s')\ko}(1+\|f\|_{\xsbt})\|g\|_{\xsbt}\|h\|_{\xsbt})^2
\end{align*}
which is acceptable when summed as in \eqref{eqn:sum2}.

When $f_{\ll\kt}$ is at high modulation, set $1/M=b'$ and use Bernstein's inequality on the spatial variable of $f$ to estimate
\begin{align*}
	\|P_kQ_l(\eta_T\T(Q_{\gtrsim l}f_{\ll \kt},g_{\kt},h_{\leq\kt})\|_{\xsobot}
	&\lesssim2^{(s'+b')l}\|\eta_T\|_M\|Q_{\gtrsim l}f_{\ll\kt}\|_{\f{2M}{M-2},\infty}\|Dg_{\kt}\|_{\infty,2}\|Dh_{\leq\kt}\|_{\infty,\infty}\\
	&\lesssim T^{b'}2^{-(2-b')l}2^{(2-2s')\kt}(1+\|f\|_{\xsbt})\|g\|_{\xsbt}\|h\|_{\xsbt}.
\end{align*}

When $g_{\kt}$ is at high modulation we use the same choice of $M$ to bound
\begin{align*}
	\|P_kQ_l(\eta_T\T(f_{\ll \kt},Q_{\gtrsim l}g_{\kt},h_{\leq\kt})\|_{\xsobot}
	&\lesssim2^{(s'+b')l}\|\eta_T\|_M\|f_{\ll\kt}\|_{\infty,\infty}\|Q_{\gtrsim l}Dg_{\kt}\|_{\f{2M}{M-2},2}\|Dh_{\leq\kt}\|_{\infty,\infty}\\
	&\lesssim T^{b'}2^{-(1-b')l}2^{(1-s')\kt}(1+\|f\|_{\xsbt})\|g\|_{\xsbt}\|h\|_{\xsbt}.
\end{align*}

Finally, when $h_{\leq\kt}$ is at high modulation, we again take $1/M=b'$ and estimate
\begin{align*}
	\|P_kQ_l(\eta_T\T(f_{\ll \kt},g_{\kt},Q_{\gtrsim l} h_{\leq\kt})\|_{\xsobot}
	&\lesssim2^{(s'+b')l}\|\eta_T\|_M\|f_{\ll\kt}\|_{\infty,\infty}\|Dg_{\kt}\|_{\infty,\infty}\|Q_{\gtrsim l}Dh_{\leq\kt}\|_{\f{2M}{M-2},2}\\
	&\lesssim T^{b'}2^{-(1-b')l}2^{(1-s')\kt}(1+\|f\|_{\xsbt})\|g\|_{\xsbt}\|h\|_{\xsbt}
\end{align*}
which is again acceptable when summed as in \eqref{eqn:sum2}.

This completes the proof of \eqref{eqn:trilinear}. The remark which follows can be verified by inspection of the proof.
\end{proof}

\section{Error representation}\label{sec:error_representation}
Henceforth fix $0<\tau<1$, and $3/2<s_1<s<2$ with $s-s_1<2-s$. Then choose $b$ and $b_1$ such that $s-s_1<b-1/2<\min\{s-3/2,2-s\}$ and $1/2<b_1<b-(s-s_1)$. Note $s+b<5/2$ and $s_1-3/2>b_1-1/2>0$.

Fix $\uu_0\in (p+H^s(\R^3))\x H^{s-1}(\R^3)$ and $\uu\in p+\XX^{s,b}_{\text{loc}}([0,T_{\text{max}}))$ the maximal solution to the wave maps equation obtained in Theorem \ref{lwp}. Fix $0<T<T_{\text{max}}$ and let $M=\|\uu\|_{\XX^{s,b}([0,T])}$ (from now on, we omit any mention of the constant $p$ to reduce notation). Further fix a global extension $\uu^g$ of $\uu$ from $[0,T]$ to $\R$ such that $\|\uu^g\|_{\XX^{s,b}(\R)}\leq 2M$.

Henceforth $\eta$ denotes a fixed smooth function equal to $1$ on $[-1,1]$ and vanishing on $[-2,2]^c$. All implicit constants may depend on $s,\,s_1,\,b,\,b_1$ and $\uu_0$. Observe that $M$ depends only on $T$ (up to implicit constants).

Let $\uu_n$ denote the iterates obtained by the scheme \eqref{SS}.
Our first goal is to show that
\begin{equation}\label{eqn:goal}
\|\uu_n-\uu(n\tau)\|_{X^{s_1,b_1}_\tau([0,T_1])}\leq C(T) \tau^{(s-s_1)/2}
\end{equation}
for some $T_1(T)$ sufficiently small. At the end of the paper we will extend this to $[0,T]$ by iteration.
We will estimate \ref{eqn:goal} via the intermediate function $\uu_\tau$, solving 
a truncated version of \eqref{WM3},
\begin{equation*}
	\dd_t\uu_\tau=
	\begin{pmatrix}
		0 & 1\\
		\D & 0
	\end{pmatrix}
	\uu_\tau+\NN_\Pi(u_\tau),\hspace{3em}\uu_\tau(0)=\Pi\uu(0).
\end{equation*}
Here $\NN_\Pi=\Pi\NN\Pi$ for $\Pi:=\Pi_{\pi/(8\tau^{1/2})}$.
This admits the Duhamel formulation
\begin{equation}\label{eqn:u_tau_duhamel}
	\uu_\tau(t)=\LL_t\Pi\uu_0+\int_0^t\LL_{t-s}\NN_\Pi(u_\tau(s))ds.
\end{equation}
We first show that $\uu_\tau$ is a suitable approximation of $\uu$. Let $C_0$ denote the maximum of the constants appearing in \eqref{item1} and its continuous analogue, and $C_1$ as the maximum of the constants in \eqref{item2} and the continuous analogue. In a minor abuse of notation we will still denote $\T(f,g,h)=-f(\Box(gh)-g\Box h-h\Box g)$ in the continuous setting.
\begin{theorem}\label{thm:truncation_theorem}
	For $0<T_1(T)<\min\{T,1\}$ sufficiently small, \eqref{eqn:u_tau_duhamel} admits a unique solution $\uu_\tau\in X^{s,b}([0,T_1])$.	
	The solution has a global extension $\uu^g_\tau$ solving
	\begin{equation}\label{eqn:u^g_tau}
		\uu^g_\tau=\eta(t)\Bigl(\LL_t\Pi\uu_0+\int_0^t\LL_{t-s}\eta_{T_1}(s)\NN_\Pi(u^g_\tau(s))\,ds\Bigr)
	\end{equation}
	with $\|\uu^g_\tau\|_{\xxsb(\R)}\leq 2C_0C_1M$.
	
	Moreover, $\uu_\tau$ converges to $\uu$ as $\tau\to0$ in the sense that
	\begin{equation*}
		\|\uu-\uu_\tau\|_{X^{s_1,b_1}([0,T_1])}\leq 2C_1\|\uu(0)-\uu_\tau(0)\|_{H^{s_1}\times H^{s_1-1}}+C_2(1+M)^2\|(1-\Pi)u^g\|_{X^{s_1,b_1}(\R)}
		\lesssim_T\tau^{(s-s_1)/2}.
	\end{equation*}
	Here $C_2>0$ is a constant depending only on the parameters $s,\,s_1,\,b$ and $b_1$, and $\uu^g$ is the global extension of $\uu$ introduced at the start of this section.
\end{theorem}
This lemma also holds for $s=s_1$.
\begin{proof}
	We obtain existence of a solution to \eqref{eqn:u^g_tau} for $T_1\lesssim \bigl(M(1+M)\bigr)^{-1/\alpha(s,b)}$ sufficiently small by Picard iteration in the space
	$$
	B_M:=\{\uu\in \XX^{s,b}(\R):\|\uu\|_{\XX^{s,b}(\R)}\leq 2C_0C_1M\}
	$$
	using the continuous trilinear estimate. It is clear that $\uu^g_\tau$ solves \eqref{eqn:u_tau_duhamel} on $[0,T_1]$; uniqueness in $\XX^{s,b}([0,T_1])$ follows by standard arguments.
	
	We now turn to the question of convergence.	For $T_1$ as above let $\uu_\tau\in \XX^{s,b}([0,T_1])$ be the unique solution to \eqref{eqn:u_tau_duhamel} and $\uu^g_\tau$ be the global extension.	
	Let $\mathbf{w}^g$ be an arbitrary global extension of $\uu-\uu_\tau$ in $\XX^{s,b}(\R)$. Then 
	\begin{align*}
		\eta(t)\Bigl(&\LL_{t}(\uu(0)-\uu_\tau(0))\\
		&+\int_0^t\LL_{t-s}\eta_{T_1}(s)\chi_{[0,\infty)}(s)[\T(u^g,u^g,u^g)-\Pi\T(\Pi u^g,\Pi u^g,\Pi u^g)]\,ds\\
		&+\int_{0}^t\LL_{t-s}\,\eta_{T_1}(s)\chi_{[0,\infty)}(s)\Pi[\mathcal{T}(\Pi w^g,\Pi u^g,\Pi u^g)+\mathcal{T}( \Pi u^g_\tau,\Pi w^g,\Pi u^g)+\mathcal{T}(\Pi u^g_\tau,\Pi u^g_\tau,\Pi w^g)]\,ds\Bigr)\\
	\end{align*}
	is also a global extension of $\uu-\uu_\tau$ from $[0,T_1]$.
	For $s_1$ and $b_1$ as in the theorem statement, we apply the linear estimates of Lemma \ref{lem:bourgain_prop} followed by the sharp time cut-off estimate in Lemma \ref{lem:ctnty-time-cut-off} to obtain 
	\begin{align*}
		&\|\mathbf{u}-\mathbf{u}_\tau\|_{\XX^{s_1,b_1}([0,T_1])}\\
		&\leq C_1\|\uu(0)-\uu_\tau(0)\|_{H^{s_1}\times H^{s_1-1}}+C\|\eta_{T_1}(s)[\T(u^g,u^g,u^g)-\Pi\T(\Pi u^g,\Pi u^g,\Pi u^g)]\|_{X^{s_1-1,b_1-1}(\R)}\\
		&\quad+C\|\eta_{T_1}(s)\bigl(\mathcal{T}( \Pi w^g, \Pi u^g, \Pi u^g)+\mathcal{T}(\Pi u^g_\tau, \Pi w^g, \Pi u^g)+\mathcal{T}(\Pi  u^g_\tau, \Pi u^g_\tau, \Pi w^g)\bigr)\|_{X^{s_1-1,b_1-1}(\R)}.
	\end{align*}
	Before applying the trilinear estimates, observe that the difference of frequency cut-offs in the second term imposes that one of the three factors $u^g$ must also be at frequency $\gtrsim \tau^{-1/2}$. Allowing some flexibility in the interpretation of $\Pi$, this leads to
	\begin{align*}
		&\|\mathbf{u}-\mathbf{u}_\tau\|_{\XX^{s_1,b_1}([0,T_1])}\\
		&\leq C_1\|\uu(0)-\uu_\tau(0)\|_{H^{s_1}\times H^{s_1-1}}+CT_1^\al\|(1-\Pi)u^g\|_{X^{s_1,b_1}(\R)}(1+\|u^g\|_{X^{s_1,b_1}(\R)})^2\\
		&\quad+CT_1^\al \|w^g\|_{X^{s_1,b_1}(\R)}(1+\|u^g\|_{X^{s_1,b_1}(\R)}+\|u^g_\tau\|_{X^{s_1,b_1}(\R)})^2\\
		&\leq C_1\|\uu(0)-\uu_\tau(0)\|_{H^{s_1}\times H^{s_1-1}}+C\|(1-\Pi)u^g\|_{X^{s_1,b_1}(\R)}(1+2M)^2\\
		&\quad+CT_1^\al\|w^g\|_{X^{s_1,b_1}(\R)}(1+4C_0C_1M)^2
	\end{align*}
	for some constant $C>0$ which may change line-to-line.
	Taking the infimum over all possible extensions $w^g$ and $T_1(T)$ slightly smaller if necessary, we obtain
	$$
	\|\mathbf{u}-\mathbf{u}_\tau\|_{\XX^{s_1,b_1}([0,T_1])}\leq 2C_1\|\uu(0)-\uu_\tau(0)\|_{H^{s_1}\times H^{s_1-1}}+C_2\|(1-\Pi)u^g\|_{X^{s_1,b_1}(\R)}(1+M)^2
	$$
	and the result follows.
\end{proof}
Henceforth we fix $T_1(T)$ and $\uu_\tau\in X^{s,b}([0,T_1])$ as above. In particular, $\|\uu_\tau\|_{X^{s,b}([0,T_1])}\leq 2C_0C_1M$.
Having suitably bounded $\|\uu-\uu_\tau\|_{X^{s_1,b_1}}$, which controls $\|\uu(n\tau)-\uu_\tau(n\tau)\|_{X^{s_1,b_1}_\tau}$ by Lemma \ref{cts_dis_lem}, it remains to estimate $\|\uu_n-\uu_\tau(n\tau)\|_{X^{s_1,b_1}_\tau}$. We first confirm that the iterates $\uu_n$ indeed belong to the Bourgain space. Write
\begin{equation}\label{eqn:true_iterates}
	\uu_n=\LL_{n\tau}\Pi\uu_0+\tau\sum_{k=0}^{n-1} \LL_{(n-k)\tau}\chi_{[2\tau,\infty)}(t_k)\NN_\tau(u_k)
\end{equation}
for $\NN_\tau$ as in \eqref{eqn:N_tau}.
Note the summation convention \eqref{eqn:summation_convention} is irrelevant in the presence of the time cut-off.
\begin{theorem}\label{thm:un_in_Xsbt} 
	For $0<T_1(T)<\min\{T,1\}$ sufficiently small the equation \eqref{eqn:true_iterates} admits a unique solution $\uu_n\in X^{s,b}_\tau([0,T_1])$. This solution coincides with the constructive iterates and has a global extension $\uu^g_n$ solving
	\begin{equation}
		\uu^g_n=\eta(t_n)\Bigl(\LL_{n\tau}\Pi\uu_0+\tau\sum_{k=0}^{n-1} \LL_{(n-k)\tau}\chi_{[2\tau,\infty)}(t_k)\eta_{T_1}(t_k)\NN_\tau(u^g_k)\Bigr)
	\end{equation}
	with $\|\uu^g_n\|_{\xxsbt(\R)}\leq 2C_0C_1M$.
\end{theorem}
While it is clearest to state in $\xxsbt$, we will only use the bound $\|\uu^g_n\|_{\XX^{s_1,b_1}_\tau(\R)}\leq 2C_0C_1M$. This is important to notice for the iteration over subintervals of $[0,T]$.
\begin{proof}
	We omit the proof which is similar to that of Theorem \ref{thm:truncation_theorem}. Note it is trivial that the iterates coincide with the solution constructed via Picard iteration on $[0,T_1]$ since the formula is constructive in that region.
\end{proof}


Now observe that for $0\leq n\tau\leq T_1$ we can write
\begin{equation}\label{duhamel}
	\uu_\tau(t_n)=\LL_{n\tau}\Pi\uu_0+\sum_{k=0}^{n-1}\LL_{(n-k)\tau}\int_0^\tau\LL_{-s}\NN_\Pi(u_\tau(t_k+s))ds,
\end{equation}
so
\begin{align*}
	\uu_n-\uu_\tau(t_n)&=\tau\sum_{k=0}^{n-1}\LL_{(n-k)\tau}\chi_{[2\tau,\infty)}(t_k)\bigl(\NN_\tau(u_k)-\NN_\tau(u_\tau(t_k))\bigr)\\
	&\quad+\sum_{k=0}^{n-1}\LL_{(n-k)\tau}\Bigl(\tau\chi_{[2\tau,\infty)}(t_k)\NN_\tau(u_\tau(t_k))-\int_0^\tau \LL_{-s}\NN_\Pi(u_\tau(t_k+s))\,ds\Bigr).
\end{align*}
Consider the equation
\begin{align}
	\mathbf{w}_n&=\eta(t_n)\tau\Pi\sum_{k=0}^{n-1}\LL_{(n-k)\tau}\eta_{T_1}(t_k)\chi_{[2\tau,\infty)}(t_k)\bigl(\T(u^g_k,u^g_k,w_k)+\T(u^g_k,w_k,u^g_\tau(t_k))+\T(w_k,u^g_\tau(t_k),u^g_\tau(t_k)\Bigr)\nonumber\\
	&\quad+\eta(t_n)\sum_{k=0}^{n-1}\LL_{(n-k)\tau}\eta_{T_1}(t_k)\chi_{[0,\infty)}(t_k)\EE_{\text{loc}}(t_k,\tau,u^g_\tau)\label{eqn:w_n}
\end{align}
where the (summand of) the local error is defined on a function $v:\R\to\R^3$ by
\begin{align}
	\EE_{\text{loc}}(t_k,\tau,v):&=-\tau \chi_{[0,2\tau)}(t_k)\NN_\tau(v(t_k))\label{eqn:local_error}\\
	&\quad+\int_0^\tau \NN_\tau(v(t_k))-\NN_\tau( v(t_k+s))\,ds\nonumber\\
	&\quad+\int_0^\tau \NN_\tau(v(t_k+s))-\NN_\Pi(v(t_k+s))\,ds\nonumber\\
	&\quad+\int_0^\tau(I-\LL_{-s}) \NN_\Pi(v(t_k+s))\,ds\nonumber\\
	&=:\mathcal{E}_1(t_k,\tau,v)+\mathcal{E}_2(t_k,\tau,v)+\mathcal{E}_3(t_k,\tau,v)+\mathcal{E}_4(t_k,\tau,v).\nonumber
\end{align}
It is clear that $\mathbf{w}_n$ defines a global extension of $\uu_n-\uu_\tau(t_n)$ from $[0,T_1]$.

The next section is dedicated to bounding the local error. Note that the local error depends only on $u^g_\tau$ so should be viewed as a forcing term in the equation for $\mathbf{w}_n$.

\section{Local error analysis}\label{sec:local-error}
The key difficulty in estimating the local error is controlling the error caused by the discrete approximation of the wave operator in the nonlinearity. Our first result for this section resolves this issue by improving Lemma \ref{key_lemma} by nonlinear estimates. Recall $s,\,s_1,\,b$ and $b_1$ are fixed as in Section \ref{sec:error_representation}.

\begin{lemma}\label{lem:nonlinear_lemma}
	For $u^g_\tau$ as in \eqref{eqn:u^g_tau},
	$$
	\tau\|\dd_t^3 u^g_\tau\|_{X^{s_1-1,b_1-1}}+\tau^{s-s_1+b}\|\dd_tu^g_\tau\|_{X^{s,b}}\lesssim \tau^{(s-s_1)/2}\|u^g_\tau\|_{X^{s,b}}(1+\|u^g_\tau\|_{X^{s,b}})^2.
	$$
\end{lemma}
\begin{proof}
	We write $u=u^g_\tau$ to reduce notation.	
	To handle the $\dd_t^3$ term, we divide the integral into the regions $|\sg|\leq2|\xi|$ and $|\sg|>2|\xi|$.
	\begin{equation}\label{eqn:region1}
	\tau\|\chi_{|\dd_t|\leq 2|\dd_x|}\dd_t^3u\|_{X^{s_1-1,b_1-1}}\lesssim \tau^{1/2}\tau^{(s-s_1)/2}\|\dd_tu\|_{X^{s,b_1-1}}\lesssim\tau^{1/2}\tau^{(s-s_1)/2}\|\dd_tu\|_{X^{s,b}}.
	\end{equation}
	On the other hand,
	\begin{equation}\label{eqn:region2}
	\tau\|\chi_{|\dd_t|>2|\dd_x|}\dd_t^3u\|_{X^{s_1-1,b_1-1}}\lesssim \tau\|\dd_tu\|_{X^{s_1,b_1}}\lesssim\tau\|\dd_tu\|_{X^{s,b}}.
	\end{equation}
	It therefore suffices to show that
	\begin{equation}\label{eqn:dt_ug}
		\|\dd_tu\|_{X^{s,b}}\lesssim \tau^{-1/2}\|u\|_{X^{s,b}}(1+\|u\|_{X^{s,b}})^2.
	\end{equation}
	
	Observe that
	\begin{align*}
		\dd_tu(t)&=\eta'(t)\Bigl(\LL_t\Pi\uu_0+\int_0^t\LL_{t-s}\eta_{T_1}(s)\NN_{\Pi}(u(s))\,ds\Bigr)\\
		&\quad+\eta(t)\Bigl(\LL_t(\dd_tu(0),\dd_t^2u(0))+\int_0^t\LL_{t-s}\bigl(\eta_{T_1}(s)\widetilde{\NN}_\Pi(u(s),\dd_tu(s))+\eta_{T_1}'(s)\NN_\Pi(u(s)\bigr))\,ds\Bigr)
	\end{align*}
	where $\widetilde{\NN}_\Pi(u,\dd_tu)=\Pi\widetilde{\NN}(\Pi u,\Pi \dd_tu)$ for
	\begin{align*}
		\widetilde{\NN}(u,\dd_tu)&=\T(\dd_tu,u,u)+\T(u,\dd_tu,u)+\T(u,u,\dd_tu).
	\end{align*}
	Therefore the linear and trilinear estimates yield
	\begin{align*}
	\|\dd_tu\|_{X^{s,b}}&\lesssim \|\uu_0\|_{H^s\x H^{s-1}}+\|\dd_tu(0)\|_{H^s}+\|\dd_t^2u(0)\|_{H^{s-1}}\\
	&\quad+T_1^{\al-1}\|u\|_{X^{s,b}}^2(1+\|u\|_{X^{s,b}})+T_1^{\al}\|\dd_tu\|_{X^{s,b}}\|u\|_{X^{s,b}}(1+\|u\|_{X^{s,b}}),
	\end{align*}
	where we used that for $u\in p+X^{s,b}$, $\dd_tu\in X^{s,b}$. Thus taking $T_1(T)$ sufficiently small, we see that it suffices to prove that
	\begin{equation}\label{eqn:NL(0)}
		\|\dd_t u(0)\|_{H^s}+\|\dd_t^2u(0)\|_{H^{s-1}}\lesssim \tau^{-1/2}\|u\|_{X^{s,b}}(1+\|u\|_{X^{s,b}})^2.
	\end{equation}
	The bound on the first derivative is immediate due to the frequency restriction and the embedding $\|\dd_t u(0)\|_{H^{s-1}}\lesssim\|\uu\|_{\XX^{s,b}}$.
	
	For the term involving the second derivative, we write
	$$
	\dd_t^2u(0)=\Delta u(0)+\NN_\Pi(u(0)).
	$$
	It is clear that $\|\Delta u(0)\|_{H^{s-1}}\lesssim \tau^{-1/2}\|u(0)\|_{H^s}$, so we only need to study the second term,
	$$
	\|u(0)\,\dd^\al u(0)\cdot\dd_\al u(0)\|_{H^{s-1}}
	$$
	(we ignore the frequency projections $\Pi$ which merely clutter the formulae).
	Fix some $M>0$ such that $s>3/2+3/M$. By the fractional Leibniz rule and Sobolev embeddings we have
	\begin{align*}
		\|u(0)\,\dd^\al u(0)\cdot\dd_\al u(0)\|_{H^{s-1}}&\lesssim \|\lan \na\ran^{s-1}u(0)\|_{L^6}\|\dd^\al u(0)\|_{L^6}\|\dd_\al u(0)\|_{L^6}+\|u(0)\|_M\|\lan \na\ran^{s-1}(\dd^\al u(0)\cdot\dd_\al u(0))\|_{L^\f{2M}{M-2}}\\
		&\lesssim \|\lan \na\ran^{s-1}u(0)\|_{H^{1}}\|\dd^\al u(0)\|_{H^1}^2+\|u(0)\|_{H^{3/2-3/M}}\|\dd_\al u(0)\|_{\f{3M}{M-3}}\|\lan \na\ran^{s-1}\dd^\al u(0)\|_{L^6}\\
		&\lesssim \|u(0)\|_{H^s}(\tau^{-\half(2-s)}\|\dd^\al u(0)\|_{H^{s-1}})^2+\|u(0)\|_{H^s}\|\dd_\al u(0)\|_{H^{1/2+3/M}}\|\dd^\al u(0)\|_{H^s}\\
		&\lesssim (\tau^{s-2}+\tau^{-1/2})\|u(0)\|_{H^s}\|\dd^\al u(0)\|_{H^{s-1}}^2.
	\end{align*}
	Since $s>3/2$, this proves \eqref{eqn:NL(0)}.
\end{proof}

	Combining the previous lemma with Lemma \ref{key_lemma} we see that
	\begin{equation}\label{eqn:key_diff_bound}
	\|(\Box-\Box_\tau)u^g_\tau(t_n)\|_{X^{s_1-1,b-1}_\tau}\lesssim_{s,s_1,b,b_1} \tau^{(s-s_1)/2}\|u^g_\tau\|_{X^{s,b}}(1+\|u^g_\tau\|_{X^{s,b}})^2.
\end{equation}

We will also need the above result for $|u^g_\tau(t_n)|^2$. Again set $u=u^g_\tau$. By Lemma \ref{key_lemma} we still have
\begin{align*}
\|(\Box-\Box_\tau)(|u(t_n)|^2)\|_{X^{s_1-1,b_1-1}_\tau}&\leq \tau\|\dd_t^3(|u|^2)\|_{X^{s_1-1,b_1-1}}+\tau^{s-s_1+b}\|\dd_t(|u|^2)\|_{X^{s,b}}.
\end{align*}
The same calculations \eqref{eqn:region1}-\eqref{eqn:region2} show that
$$
\tau\|\dd_t^3(|u|^2)\|_{X^{s_1-1,b_1-1}}\lesssim \tau^{1/2}\tau^{(s-s_1)/2}\|\dd_t(|u|^2)\|_{X^{s,b}}.
$$
so it remains to estimate $\|\dd_t(|u|^2)\|_{X^{s,b}}$. 
For this we use the bilinear estimate \eqref{eqn:algebra-property}, which in conjunction with \eqref{eqn:dt_ug} shows that
$$
\|\dd_t(|u|^2)\|_{X^{s,b}}\lesssim\|u\|_{X^{s,b}}\|\dd_t u\|_{X^{s,b}}\lesssim \tau^{-1/2}\|u\|_{X^{s,b}}^2(1+\|u\|_{X^{s,b}})^2.
$$
We have thus shown that
	\begin{equation}\label{eqn:key_diff_bound2}
	\|(\Box-\Box_\tau)(|u_n|^2)\|_{X^{s_1-1,b-1}_\tau}\lesssim\tau^{(s-s_1)/2}\|u\|_{X^{s,b}}^2(1+\|u\|_{X^{s,b}})^2.
\end{equation}

We can now bound the local error.
\begin{lemma}\label{lem:local_error_bound}
	There exists a constant $C_{T_1}>0$ such that
	$$
	\tau^{-1}\|\eta_{T_1}(t_n)\EE_{\text{loc}}(t_n,\tau,u^g_\tau)\|_{X^{s_1-1,b_1-1}_\tau}\leq C_{T_1}\tau^{(s-s_1)/2}(1+\|u^g_\tau\|_{X^{s,b}})^5.	
	$$
\end{lemma}
\begin{proof}
	Note that $s_1-3/2>b_1-1/2>0$. Observe that for any sequence $v_n$, an application of Theorem \ref{thm:bilinear_estimates} with Lemma \ref{lem4} shows that
	\begin{align}
		\|\NN_\tau(v_n)\|_{X^{s-1,b-1}_\tau}\lesssim\|\Pi v_n\|_{X^{s,b}_\tau}\|\Box_\tau (\Pi v_n\cdot\Pi v_n)\|_{X^{s-1,b-1}_\tau}+\|\Pi v_n\|_{X^{s,b}_\tau}^2\|\Box \Pi v_n\|_{X^{s-1,b-1}_\tau} 
		&\lesssim \|v_n\|_{X^{s,b}_\tau}^3.\label{nl_estimate}
	\end{align}
	
	First consider $\EE_1$. By the sharp cut-off estimate of Lemma \ref{lem:ctnty-time-cut-off} followed by \eqref{nl_estimate} and Lemma \ref{cts_dis_lem}, we bound
	\begin{align*}
		\tau^{-1}\|\eta_{T_1}(t_n)\EE_1(t_n,\tau,u^g_\tau)\|_{X^{s_1-1,b_1-1}_\tau}&\lesssim\|\chi_{[0,2\tau)}(t_n)\NN_\tau(u^g_\tau(t_n))\|_{X^{s_1-1,b_1-1}_\tau}\\
		&\lesssim\tau^{s-s_1}\|\NN_\tau(u^g_\tau(t_n))\|_{X^{s-1,b-1}_\tau}\\
		&\lesssim\tau^{s-s_1}\|u^g_\tau\|_{X^{s,b}}^3.
	\end{align*}
	
	The second term, $\EE_2$, is the most delicate. We estimate
	\begin{align}
		\tau^{-1}\|\eta_{T_1}(t_n)\EE_2(t_n,\tau,u^g_\tau)\|_{X^{s_1-1,b_1-1}_\tau}
		&\lesssim \tau^{-1}\int_0^\tau\|\eta_{T_1}(t_n)\bigl(\NN_\tau(u^g_\tau(t_n))-\NN_\tau(u^g_\tau(t_n+s))\bigr)\|_{X^{s_1-1,b_1-1}_\tau}ds\nonumber
	\end{align}
	where
	\begin{align}
	&\NN_\tau(u^g_\tau(t_n))-\NN_\tau(u^g_\tau(t_n+s))\nonumber\\
	&=-\Pi\Bigl[\Pi\bigl(u^g_\tau(t_n)-u^g_\tau(t_n+s)\bigr)\,\bigl(\Box_\tau(|\Pi u^g_\tau(t_n)|^2)-2\Pi u^g_\tau(t_n)\cdot\Box_\tau\Pi u^g_\tau(t_n)\bigr)\nonumber\\
	&\hspace{4em}+\Pi u^g_\tau(t_n+s)\,\Box_\tau\bigl(\bigl(\Pi u^g_\tau(t_n)+\Pi u^g_\tau(t_n+s)\bigr)\cdot\bigl(\Pi u^g_\tau(t_n)-\Pi u^g_\tau(t_n+s)\bigr)\bigr)\nonumber\\
	&\hspace{4em}-2\Pi u^g_\tau(t_n+s)\,\Pi\bigl(u^g_\tau(t_n)-u^g_\tau(t_n+s)\bigr)\cdot\Box_\tau\Pi u^g_\tau(t_n)\nonumber\\
	&\hspace{4em}-2\Pi u^g_\tau(t_n+s)\,\Pi u^g_\tau(t_n+s)\cdot\Box_\tau\Pi \bigl(u^g_\tau(t_n)-u^g_\tau(t_n+s)\bigr)\Bigr]\label{B-1}
	\end{align}
The difference $u^g_\tau(t_n)-u^g_\tau(t_n+s)$ is estimated in the following claim.
\begin{claim}
For $|s|<\tau$ it holds
\begin{equation*}
		\|\eta_{T_1}(t_n)\bigl(u^g_\tau(t_n)-u^g_\tau(t_n+s)\bigr)\|_{\xsobot}\lesssim_{T_1}\tau^{s-s_1}\|u^g_\tau\|_{X^{s,b}}(1+\|u^g_\tau\|_{X^{s,b}}^2).
	\end{equation*}
\end{claim}
\begin{proof}[Proof of claim.]
Write
	\begin{align}\label{eqn:true_diff}
		u^g_\tau(t_n)-u^g_\tau(t_n+s)&=(1-\LL_s)u^g_\tau(t_n)\\
		&\quad+\bigl(\eta(t_n)-\eta(t_n+s)\bigr)\LL_s\bigl(\LL_{t_n}\Pi u_0+\int_0^{t_n}\LL_{t_n-\la}\eta_{T_1}(\la)\NN_\Pi(u^g_\tau(\la))d\la\bigr)\\
		&\quad-\eta(t_n+s)\int_0^s\LL_{s-\la}\eta_{T_1}(t_n+\la)\NN_\Pi(u^g_\tau(t_n+\la))d\la.
	\end{align}
	The first term is acceptable thanks to Lemma \ref{linear_lemma}. For the second term we use Lemma \ref{cts_dis_lem} to estimate
	\begin{align*}
	&\Bigl\|\eta_{T_1}(t_n)\bigl(\eta(t_n)-\eta(t_n+s)\bigr)\LL_s\bigl(\LL_{t_n}\Pi_0+\int_0^{t_n}\LL_{t_n-\la}\eta_{T_1}(\la)\NN_\Pi( u^g_\tau(\la))d\la\bigr)\Bigr\|_{\xsobot}\\
	&\lesssim\int_0^s\Bigl\|\eta_{T_1}(t_n)\eta'(t+\tht)\LL_s\Bigl(\LL_{t}\Pi \uu_0+\int_0^{t}\LL_{t-\la}\eta_{T_1}(\la)\NN_\Pi( u^g_\tau(\la))d\la\Bigr)\Bigr\|_{X^{s_1,b_1}}d\tht\\
	&\lesssim\tau \|\eta_{T_1}(t_n)\LL_s\Bigl(\LL_{t}\Pi \uu_0+\int_0^{t}\LL_{t-\la}\eta_{T_1}(\la)\NN_\Pi( u^g_\tau(\la))d\la\Bigr)\Bigr\|_{X^{s_1,b_1}}.
	\end{align*}
	Then we use the second statement in Lemma \ref{linear_lemma} and the trilinear estimates to obtain
	\begin{align*}
	&\Bigl\|\eta_{T_1}(t_n)\bigl(\eta(t_n)-\eta(t_n+s)\bigr)\LL_s\bigl(\LL_{t_n}\Pi_0+\int_0^{t_n}\LL_{t_n-\la}\eta_{T_1}(\la)\NN_\Pi( u^g_\tau(\la))d\la\bigr)\Bigr\|_{\xsobot}\\
	&\lesssim\tau\|\eta_{T_1}(t_n)\Bigl(\LL_{t}\Pi u_0+\int_0^{t}\LL_{t-\la}\eta_{T_1}(\la)\NN_\Pi( u^g_\tau(\la))d\la\Bigr)\Bigr\|_{X^{s_1,b_1}}\\
	&\lesssim\tau\bigl(\|\uu_0\|_{H^s\x H^{s-1}}+T_1^\al(1+\|u^g_\tau\|_{X^{s,b}})\|u^g_\tau\|_{X^{s,b}}^2\bigr)
	\end{align*}
	which is more than acceptable since $T_1<1$.
	
	It remains to handle the more delicate final term. Consider a function $N(t,x)$ and the sequence $N_n^\la(x):=N(n\tau+\la,x)$. Set $b':=b-(s-s_1)$. The assumptions on $s,\,s_1,\,b$ and $b_1$ imply that $s_1-3/2>b'-1/2>0$. On the one hand, by Lemmas \ref{linear_lemma} and \ref{cts_dis_lem},
	\begin{align}\label{B1}
		\bigl\|\eta_{T_1}(t_n)\int_0^s\LL_{s-\la}N_n^\la d\la\bigr\|_{X^{s_1,b'}_\tau}\lesssim_{T_1}\int_0^s\|N_n^\la\|_{X^{s_1,b'}_\tau}d\la
		&\lesssim_{T_1}\tau \sup_{|\la|<\tau}\|N_n^\la\|_{X^{s_1,b'}_\tau}\lesssim_{T_1}\tau\|N\|_{X^{s_1,b'}}.
	\end{align}
	On the other hand, by the continuous version of \eqref{item3} we have
	\begin{align}
		\bigl\|\eta_{T_1}(t_n)\int_0^s\LL_{s-\la}N_n^\la d\la\bigr\|_{X^{s_1,b'}_\tau}&\lesssim\bigl\|\eta_{T_1}(t)\int_0^s\LL_{s-\la}N(\cdot+\la,x) d\la\bigr\|_{X^{s_1,b'}}
		\lesssim_{T_1} \|N\|_{X^{s_1-1,b'-1}_\tau}.\label{B2}
	\end{align}
	By the Stein-Weiss interpolation theorem \cite[Theorem 5.4.1]{bergh2012interpolation} we deduce that
	$$
	\bigl\|\eta_{T_1}(t_n)\int_0^s\LL_{s-\la}N_n^\la d\la\bigr\|_{X^{s_1,b'}_\tau}\lesssim_{T_1}\tau^{s-s_1}\|N\|_{\xsbm}.
	$$
	Now setting $N(t,x)=\eta_{T_1}(t)\NN_\Pi(u^g _\tau)(t,x)$ and applying the continuous nonlinear estimates we obtain
	\begin{equation*}
		\|\eta_{T_1}(t_n)\eta(t_n+s)\int_0^s\LL_{s-\la}\eta_{T_1}(t_n+\la)\NN_\Pi(u^g _\tau)(t_n+\la)d\la\|_{X^{s_1,b'}_\tau}\lesssim_{T_1}
		\| u^g _\tau\|_{X^{s,b}}^3.
	\end{equation*}
	Since we assume $b_1<b'$, this completes the proof of the claim.
\end{proof}

We now return to \eqref{B-1}. For the first line, we use the bilinear estimates of Section \ref{sec:bilinear-estimates} and Lemma \ref{lem4} to bound
	\begin{align*}
		&\tau^{-1}\int_0^\tau\|\eta_{T_1}(t_n)\bigl(\Pi(u^g_\tau(t_n)-u^g_\tau(t_n+s))\,\bigl(\Box_\tau(|\Pi u^g_\tau(t_n)|^2)-2\Pi u^g_\tau(t_n)\cdot\Box_\tau\Pi u^g_\tau(t_n)\bigr)\bigr)\|_{\xsobotm}ds\\
		&\lesssim\sup_{|s|\leq\tau}\|\eta_{T_1}(t_n)(u^g_\tau(t_n)-u^g_\tau(t_n+s))\|_{\xsobot}\|\Box_\tau(|\Pi u^g_\tau(t_n)|^2)-2\Pi u^g_\tau(t_n)\cdot\Box_\tau\Pi u^g_\tau(t_n)\|_{\xsobotm}\\
		&\lesssim_{T_1}\tau^{s-s_1}\| u^g _\tau\|_{X^{s,b}}(1+\| u^g _\tau\|_{X^{s,b}})^2(\||\Pi u^g_\tau(t_n)|^2\|_{X^{s_1,b_1}_\tau}+\|u^g_\tau(t_n)\|_{X^{s_1,b_1}_\tau}\|\Box_\tau u^g_\tau(t_n)\|_{X^{s_1-1,b_1-1}_\tau})\\
		&\lesssim_{T_1}\tau^{s-s_1}\| u^g _\tau\|_{X^{s,b}}^3(1+\| u^g _\tau\|_{X^{s,b}})^2
	\end{align*}
	
	The remaining three terms in \eqref{B-1} can be treated analogously.
	
	For $\EE_3$ we write
	\begin{align*}
		&\NNt(u^g_\tau(t_n+s))-\NN_\Pi(u^g_\tau(t_n+s))\\
		&=-\Pi\Bigl(\Pi  u^g _\tau(t_n+s)\,\bigl((\Box_\tau-\Box)(|\Pi  u^g _\tau(t_n+s)|^2)-2 \Pi  u^g _\tau(t_n+s)\cdot (\Box_\tau-\Box) \Pi u^g_\tau(t_n+s)\bigr)\Bigr).
	\end{align*}
	We apply the bilinear estimates followed by \eqref{eqn:key_diff_bound} and \eqref{eqn:key_diff_bound2} to estimate
	\begin{align*}
		\tau^{-1}\|\eta_{T_1}(t_n)\EE_3(t_n,\tau,u^g_\tau)\|_{X^{s_1-1,b_1-1}_\tau}
		&\lesssim\sup_{s\in[0,\tau]}\bigl(\|\Pi u^g_\tau(t_n+s)\|_{X^{s_1,b_1}_\tau}\|(\Box_\tau-\Box)(|\Pi u^g_\tau(t_n+s)|^2)\|_{X^{s_1-1,b_1-1}_\tau}\\
		&\hspace{6em}+\|\Pi u^g_\tau(t_n+s)\|_{X^{s_1,b_1}_\tau}^2\|(\Box_\tau-\Box)\Pi u^g_\tau(t_n+s)\|_{X^{s_1-1,b_1-1}_\tau}\bigr)\\
		&\lesssim\tau^{(s-s_1)/2}\|u^g_\tau\|_{X^{s,b}}^3(1+\|u^g_\tau\|_{X^{s,b}})^2.
	\end{align*}
	
	We finally turn to $\EE_4$. By Lemma \ref{linear_lemma} we have
	\begin{align*}
		\tau^{-1}\|\eta_{T_1}(t_n)\EE_4(t_n,\tau,u^g_\tau)\|_{\xsobotm}&\lesssim\sup_{s\in[0,\tau]}\|(I-\LL_{-s})\NN_\Pi(u^g_\tau(t_n+s))\|_{\xsobotm}\\
		&\lesssim\tau^{s-s_1}\sup_{s\in[0,\tau]}\|\NN_\Pi(u^g_\tau(t_n+s))\|_{\xsbtm}\\
		&\lesssim\tau^{s-s_1}\|u^g_\tau\|_{X^{s,b}}^3.\qedhere
	\end{align*}
\end{proof}

\section{Proof of the main theorem}\label{section:main_theorem}
We now bound $\|\uu_n-\uu_\tau(t_n)\|_{X^{s_1,b_1}([0,T_1])}$. Recall the global extension $\mathbf{w}_n$ of $\uu_n-\uu_\tau(t_n)$ defined in \eqref{eqn:w_n}. The linear Bourgain space estimates show that
\begin{align*}
\|\mathbf{w}_n\|_{\XX^{s_1,b_1}_\tau(\R)}&\lesssim \|\eta_{T_1}(t_k)\chi_{[2\tau,\infty)}(t_k)\bigl(\T(u^g_k,u^g_k,w_k)+\T(u^g_k,w_k,u^g_\tau(t_k))+\T(w_k,u^g_\tau(t_k),u^g_\tau(t_k)\bigr)\|_{X^{s_1-1,b_1-1}_\tau(\R)}\\
&\quad+\tau^{-1}\|\eta_{T_1}(t_k)\chi_{[0,\infty)}(t_k)\EE_{\text{loc}}(t_k,\tau,u^g_\tau)\|_{X^{s_1-1,b_1-1}_\tau(\R)}.
\end{align*}
Thus using Lemma \ref{lem:ctnty-time-cut-off} to discard of the sharp cut-offs, followed by the trilinear estimates and Lemma \ref{lem:local_error_bound}, we obtain
\begin{align*}
\|\mathbf{w}_n\|_{\XX^{s_1,b_1}_\tau(\R)}&\lesssim T_1^\al (1+\|u^g_n\|_{X^{s_1,b_1}_\tau(\R)}+\|u^g_\tau(t_n)\|_{X^{s_1,b_1}_\tau(\R)})^2\|w_n\|_{X^{s_1,b_1}_\tau(\R)}+C_{T_1}\tau^{(s-s_1)/2}(1+\|u^g_\tau\|_{X^{s,b}(\R)})^5\\
&\lesssim T_1^\al(1+M)^2\|\mathbf{w}_n\|_{\XX^{s_1,b_1}_\tau(\R)}+C_{T_1}\tau^{(s-s_1)/2}(1+M)^5.
\end{align*}
Choosing $T_1(T)$ sufficiently small we deduce that
\begin{equation}\label{eqn:base_case}
\|\uu_n-\uu_\tau(n\tau)\|_{\XX^{s_1,b_1}_\tau([0,T_1])}\leq\|\mathbf{w}_n\|_{\XX^{s_1,b_1}_\tau(\R)}\leq C_T\tau^{(s-s_1)/2}(1+M)^5.
\end{equation}
Combined with Theorem \ref{thm:truncation_theorem} and the embedding \eqref{item1}, this proves Theorem \ref{thm:main_thm} on the interval $[0,T_1]$.

The extension to $[0,T]$ follows by a standard induction argument. The uniform bound $\|\uu\|_{\XX^{s,b}([0,T])}\leq M$ allows us to choose $T_1(M)$ uniformly small and $\tau$ sufficiently small depending on $T_1$ to obtain uniform bounds on $\|\uu_\tau\|_{\XX^{s,b}}$ and $\|\uu_n\|_{\XX^{s_1,b_1}_\tau}$ over each subinterval and so iterate the argument. Note to obtain uniform bounds on $\|\uu_\tau\|_{\XX^{s,b}}$, one must use the $s=s_1$ version of Theorem \ref{thm:truncation_theorem}, where the convergence is not quantifiable.

\section{Numerical experiments}\label{section:numerics}
Below we present the results of a numerical implementation of our scheme. 
For the spatial approximation we work on the torus $[-10,10]^3$ with a spatial resolution $\Delta x=20/2^8$, corresponding to a frequency resolution $\Delta k=\pi/10$. All spatial approximations are performed via the fast Fourier transform and the error is measured at time $T=0.5$.

For comparison with an explicit solution, we take initial data contained in a geodesic. Precisely,
$$
u_0=(\cos\tht_0(x),\sin\tht_0(x),0),\quad v_0=(0,0,0).
$$
To fix the regularity we define $\tht_0(x)\in H^{s} \backslash H^{s+\eps}$ (for any $\eps>0$) via its Fourier transform
$$
\widehat{\tht}_0(\mathbf{k})=\lan\mathbf{k}\ran^{-(s+3/2)}(\log(2+|\mathbf{k}|^2))^{-1}.
$$
Then there is the exact solution
$$
u(t,x)=\bigl(\cos\tht(t,x),\sin\tht(t,x),0\bigr)
$$
where $\tht(t,x)$ is the solution to the homogeneous linear wave equation with data $\tht_0(x)$.

In Figure \ref{fig1} we test our scheme on smooth data ($s=3$), and observe that the error decays essentially the predicted rate of $O(\tau)$ in $L^2\times H^{-1}$.

In figure \ref{fig2} we test rough initial data in $H^{1.7}\x H^{0.7}$. Here we predict the error in $H^{1.6}\x H^{0.6}$ to decay at a rate $\tau^{0.05}$. The observed decay is stronger than predicted, although naturally much slower than the decay for smooth data. 
\begin{figure}[h]
	\centering
	\begin{subfigure}{.5\textwidth}
		\centering
		\includegraphics[scale=0.25]{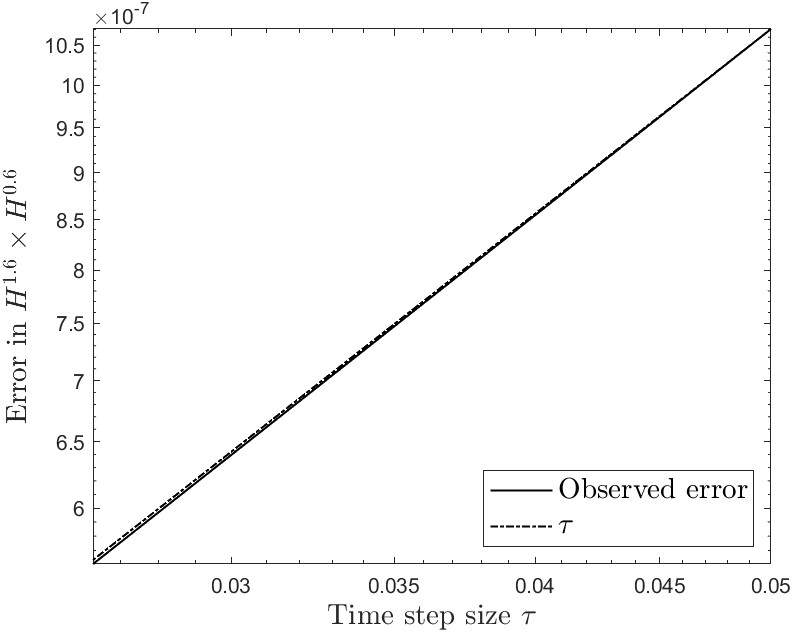}
		\caption{Convergence rate for smooth data.}
		\label{fig1}
	\end{subfigure}%
	\begin{subfigure}{.5\textwidth}
		\centering
		\includegraphics[scale=0.25]{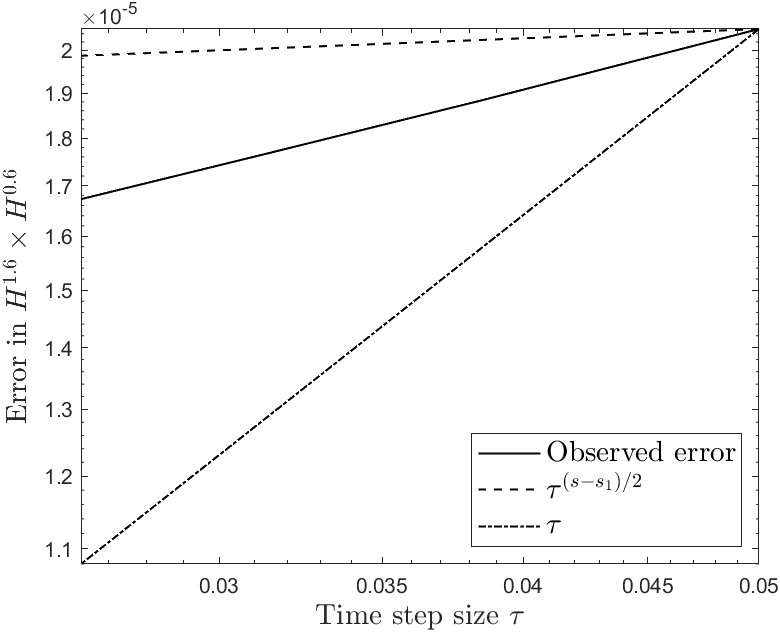}
		\caption{Convergence rate for rough data.}
		\label{fig2}
	\end{subfigure}
	\caption{}
\end{figure}




\bibliography{refs.SS} 
\end{document}